\documentclass[11pt]{article}
\usepackage{amsthm, amsmath, amssymb, amsfonts, url, booktabs, tikz, setspace, fancyhdr, bm, mathrsfs}
\usepackage{hyperref}
\usepackage{geometry}
\usepackage{hyperref, enumerate}
\usepackage[shortlabels]{enumitem}
\usepackage[babel]{microtype}
\usepackage[english]{babel}
\usepackage[capitalise]{cleveref}
\usepackage{comment}
\usepackage{bbm}
\usepackage{csquotes}
\usepackage{mathabx}
\usepackage{tikz}
\usepackage{graphicx}
\usepackage{float}
\usepackage{amsmath}

\counterwithin{figure}{section}

\newtheorem{theorem}{Theorem}[section]

\newtheorem{lemma}[theorem]{Lemma}

\newtheorem{conj}[theorem]{Conjecture}
\newtheorem{claim}[theorem]{Claim}

\newtheorem{fact}[theorem]{Fact}

\theoremstyle{definition}

\newtheorem*{defn-non}{Definition}

\newtheorem{ques}[theorem]{Question}

\newlist{Case}{enumerate}{3}
\setlist[Case, 1]{%
    label           =   {\bfseries Case \arabic*.},
    labelindent=1em ,labelwidth=1cm, labelsep*=1em, leftmargin =!
}
\setlist[Case, 2]{%
    label           =   {\bfseries Subcase \arabic{Casei}.\arabic*.},
    labelindent=-1em ,labelwidth=1cm, labelsep*=1em, leftmargin =!
}
\setlist[Case, 3]{%
    label           =   {\bfseries Subsubcase \arabic{Casei}.\arabic{Caseii}.\arabic*.},
    labelindent=-1em ,labelwidth=1cm, labelsep*=1em, leftmargin =!
}

\newenvironment{poc}{\begin{proof}[Proof of claim]}{\end{proof}}

\newcommand{\TT}{\boldsymbol{T}}
\newcommand{\RR}{\boldsymbol{R}}
\usepackage{todonotes}

\newcommand*{\abs}[1]{\lvert#1\rvert}

\newcommand{\bT}{\boldsymbol{T}}
\newcommand{\bR}{\boldsymbol{R}}
\newcommand{\bv}{\boldsymbol{v}}
\newcommand{\bA}{\boldsymbol{A}}

\newcommand{\cF}{\mathcal{F}}
\newcommand{\cG}{\mathcal{G}}

\newcommand{\cT}{\mathcal{T}}

\title{Uniform set systems with small VC-dimension}

\author{
Ting-Wei Chao\thanks{
Department of Mathematics, Massachusetts Institute of Technology, Cambridge, MA, USA, Email: twchao@mit.edu}
\and 
Zixiang Xu\thanks{Extremal Combinatorics and Probability Group (ECOPRO), Institute for Basic Science (IBS), Daejeon, South Korea. Email: zixiangxu@ibs.re.kr. Supported by IBS-R029-C4.}
\and 
Chi Hoi Yip\thanks{School of Mathematics, Georgia Institute of Technology, Atlanta, USA. Email: cyip30@gatech.edu}
\and
Shengtong Zhang\thanks{Department of Mathematics, Stanford University, USA. Email address: stzh1555@stanford.edu}
}

\begin{document}
\maketitle

\begin{abstract}

We investigate the longstanding problem of determining the maximum size of a \((d+1)\)-uniform set system with VC-dimension at most \(d\). Since the seminal 1984 work of Frankl and Pach, which established the elegant upper bound \(\binom{n}{d}\), this question has resisted significant progress. The best-known lower bound is \(\binom{n-1}{d} + \binom{n-4}{d-2}\), obtained by Ahlswede and Khachatrian, leaving a substantial gap of $\binom{n-1}{d-1}-\binom{n-4}{d-2}$. Despite decades of effort, improvements to the Frankl--Pach bound have been incremental at best: Mubayi and Zhao introduced an \(\Omega_d(\log{n})\) improvement for prime powers \(d\), while Ge, Xu, Yip, Zhang, and Zhao achieved a gain of 1 for general \(d\).

In this work, we provide a purely combinatorial approach that significantly sharpens the Frankl--Pach upper bound. Specifically, for large $n$, we demonstrate that the Frankl--Pach bound can be improved to \(\binom{n}{d} - \binom{n-1}{d-1} + O_d(n^{d-1 - \frac{1}{4d-2}})=\binom{n-1}{d}+O_d(n^{d-1 - \frac{1}{4d-2}})\). This result completely removes the main term $\binom{n-1}{d-1}$ from the previous gap between the known lower and upper bounds. It also offers fresh insights into the combinatorial structure of uniform set systems with small VC-dimension. In addition, the original Erd\H{o}s--Frankl--Pach conjecture, which sought to generalize the EKR theorem in the 1980s, has been disproven. We propose a new refined conjecture that might establish a sturdier bridge between VC-dimension and the EKR theorem, and we verify several specific cases of this conjecture, which is of independent interest. 
\end{abstract}

\section{Introduction}
\subsection{Background}
Extremal set theory concerns the problem of determining the maximum or the minimum size of a set system with certain properties. This area of research originated in 1961 with the seminal work of Erd\H{o}s, Ko, and Rado~\cite{1961EKR}, which states that for all $n\ge 2k$, the maximum size of a $k$-uniform set system on $[n]$ in which any two members have non-empty intersection, is given by $\binom{n-1}{k-1}$; moreover, when $n\geq 2k+1$, all maximum intersecting families are trivially intersecting. This result laid the groundwork for a vibrant and extensive field of study. Over time, their theorem has been significantly generalized to encompass a variety of set systems, including those with subsets of varying sizes, families satisfying broader intersection conditions, and even analogues in algebraic and geometric  contexts~\cite{2024JEMS,2020FFCOnj,2021ADVEC}. 

In this paper, we focus on problems involving the \emph{Vapnik-Chervonenkis dimension} (\emph{VC-dimension}), which is a classical complexity measure for set systems and hypothesis classes in machine learning. Formally, let \(\mathcal{F}\) be a family of subsets of a ground set \(X\). The VC-dimension of \(\mathcal{F}\), denoted by \(\mathrm{VC}(\mathcal{F})\), is the largest integer \(d\) such that there exists a set \(S = \{x_1, x_2, \dots, x_d\} \subseteq X\) with the following property:  
\[
\text{For every subset } A \subseteq S, \text{ there exists } F \in \mathcal{F} \text{ such that } F \cap S = A.
\]  
In other words, the family \(\mathcal{F}\) shatters \(S\) if it can realize all \(2^d\) possible subsets of \(S\). In particular, if no such \(S\) of size \(d+1\) exists, then \(\mathrm{VC}(\mathcal{F}) \leq d\).  

The foundational Sauer-Shelah-Perles lemma was provided by three different groups of researchers~\cite{1972JCTASauer,1972PACJMShelah,1971TPAVCVC}, who completely determined the maximum size of a set system with a given VC-dimension.

\begin{lemma}\label{lem:SSLemma}
    Let $\mathcal{F}\subseteq 2^{V}$ be a set system with VC-dimension at most $d$, then
$|\mathcal{F}|\le\sum\limits_{i=0}^{d}\binom{|V|}{i}$.
\end{lemma}
Note that the Hamming ball of radius $d$ has VC-dimension at most $d$, therefore the problem for non-uniform set systems has already been completely solved. However, the corresponding question for uniform set systems remains open.
\begin{ques}\label{question}
   What is the maximum size of $\mathcal{F}\subseteq\binom{[n]}{d+1}$ with VC-dimension at most \(d\)? 
\end{ques}
When $n\geq 2d+1$, the star \(\binom{[n-1]}{d} \cup \{n\} \subseteq \binom{[n]}{d+1}\) has VC-dimension \(d\), and it is a maximum intersecting family by the Erd\H{o}s-Ko-Rado theorem. To see the connection between Question~\ref{question} and the Erd\H{o}s-Ko-Rado theorem, observe that if $\mathcal{F}=\{F_1,F_2,\ldots, F_m\} \subset \binom{[n]}{d+1}$, then the condition of VC$(\mathcal{F})\leq d$ can be interpreted as follows: for each $1\leq i \leq m$, there exists a subset $B_{i}\subsetneq F_{i}$ such that $F\cap F_{i}\neq B_{i}$ for any $F\in\mathcal{F}$. Thus, if $\mathcal{F}$ is an intersecting family, then $B_i$ can be taken to be the empty set for all $i$. 

In the 1980s, Erd\H{o}s~\cite{1984Erdos} and Frankl and Pach~\cite{1984Franklpach} conjectured that \(\binom{n-1}{d}\) is the correct answer to Question~\ref{question} when $n$ is sufficiently large compared to $d$. If true, this conjecture would generalize the Erd\H{o}s-Ko-Rado Theorem. However, in 1997, Ahlswede and Khachatrian~\cite{1997CombFan} disproved this conjecture by constructing a \((d+1)\)-uniform set system with VC-dimension \(d\) and size \(\binom{n-1}{d} + \binom{n-4}{d-2}\) for \(n \geq 2(d+1)\). Later, Mubayi and Zhao~\cite{2007JAC} provided infinitely many non-isomorphic \((d+1)\)-uniform set systems achieving size \(\binom{n-1}{d} + \binom{n-4}{d-2}\), and conjectured that this is the correct answer to~\cref{question}.

 On the other hand, Frankl and Pach~\cite{1984Franklpach} provided an upper bound via an elegant algebraic proof.
 \begin{theorem}[\cite{1984Franklpach}]\label{thm:FranklPach}
   Let $n,d$ be positive integers with $n\ge d+1$. If $\mathcal{F}\subseteq\binom{[n]}{d+1}$ is a set system with VC-dimension at most $d$, then $|\mathcal{F}|\le\binom{n}{d}$.
\end{theorem}
 Later, Mubayi and Zhao~\cite{2007JAC} showed that when $d$ is a prime power and $n$ is sufficiently large compared to $d$, the Frankl-Pach upper bound can be improved to $\binom{n}{d} - \Omega_{d}(\log{n})$. In a recent development~\cite{2024FranklPach}, Ge, Zhao, and three of the authors showed that the Frankl--Pach upper bound can be improved to $\binom{n}{d} - 1$ for all positive integers $n,d$ with $d\ge 2$ and $n\geq 2d+2$. These results, while having elegant proofs, reveal how challenging it has been to make meaningful advances.
\subsection{Improvement on the Frankl--Pach bound}
The main term of the current gap between the upper and lower bounds is \(\binom{n-1}{d-1}\), and the aforementioned improvements in~\cite{2024FranklPach,2007JAC} to the upper bound have been relatively modest. The main result of this paper is a significant improvement of the longstanding Frankl--Pach upper bound, eliminating the main term \(\binom{n-1}{d-1}\) from the gap between lower and upper bounds.

\begin{theorem}\label{thm:GeneralD}
    Let $d\ge 2$ be a positive integer and $n$ be a sufficiently large integer compared to $d$. If $\mathcal{F}\subseteq\binom{[n]}{d+1}$ is a set system with VC-dimension at most $d$, then $|\mathcal{F}|\le\binom{n-1}{d}+C_{d}\cdot n^{d-1-\frac{1}{4d-2}}$ for some constant $C_{d}>0$.
\end{theorem}
Our result serves as additional evidence that the correct answer to \cref{question} is \(\binom{n - 1}{d} + \binom{n-4}{d-2}\); see~\cref{sec:main-proof} and~\cref{sec:conclusion} for related discussions. All previous upper bounds~\cite{1984Franklpach,2024FranklPach,2007JAC} for~\cref{question} were derived using linear algebra methods. Surprisingly, our proof of~\cref{thm:GeneralD} employs a purely combinatorial approach instead. We present the entire proof in \cref{sec:main-proof}.

We still believe that linear algebra methods have significant potential for addressing this problem and may even provide a complete resolution. In \cref{sec:PolynomialMethod}, we refine the polynomial method used in \cite{2024FranklPach}. This refinement leads to \cref{thm:MainPoly}, which provides a polynomial improvement over \cref{thm:FranklPach} by \(\Omega_{d}(n^{d-4})\). While this result is weaker than \cref{thm:GeneralD}, we believe that its proof is of independent interest.

\subsection{A modified conjecture generalizing the Erd\H{o}s-Ko-Rado theorem}
As mentioned earlier, one of the original motivations for studying \cref{question}, initiated by Erd\H{o}s~\cite{1984Erdos}, Frankl and Pach~\cite{1984Franklpach} in the 1980s, was to generalize the celebrated Erd\H{o}s-Ko-Rado theorem~\cite{1961EKR}. However, this attempt was refuted by Ahlswede and Khachatrian~\cite{1997CombFan}.  

Here we explore a different approach to generalizing the Erd\H{o}s-Ko-Rado theorem, this time in the context of VC-dimension theory. Specifically, we propose the following ``uniform Erd\H{o}s--Frankl--Pach'' conjecture as a potential correct form.
\begin{conj}\label{conj:TrueGeneralization}
   Let $n\ge 2(d+1)$ and let \(0 \leq s \leq d\). Assume $\mathcal{F} \subseteq \binom{[n]}{d+1}$ is a set family such that for every set \(F \in \mathcal{F}\), there exists a subset \(B_F \subseteq F\) of size \(s\) such that \(F \cap F' \neq B_F\) for all \(F' \in \mathcal{F}\), then \(|\mathcal{F}|\le\binom{n-1}{d}\).
\end{conj}
  In this paper, we verify this conjecture for several specific cases in \cref{sec:EKR}.
\begin{theorem}\label{thm:Verification}  
\cref{conj:TrueGeneralization} holds in the following cases:  
\begin{itemize}  
    \item $s \in \{0, d\}$, or
    \item $s=1$, and \(n\) is sufficiently large compared to $d$.
\end{itemize}  
\end{theorem}  
Note that~\cref{conj:TrueGeneralization} recovers the Erd\H{o}s-Ko-Rado theorem as the special case $s=0$. In particular, if $\mathcal{F}$ is a star, then for each $F\in \mathcal{F}$, we can take $B_F$ to be an arbitrary subset of $F \setminus \{a\}$ with size $s$; then clearly $F \cap F' \neq B_F$ for all $F, F'\in \mathcal{F}$ and the upper bound $\binom{n-1}{d}$ on $\mathcal{F}$ can be achieved. However, it appears that no other obvious family satisfies the assumption of the conjecture and achieves the same size. Thus, in the spirit of the Hilton-Milner theorem~\cite{1967Hilton}, it is plausible to conjecture that all maximum families satisfying the assumption of~\cref{conj:TrueGeneralization} are stars provided that $n$ is sufficiently large compared to $d$. Indeed, we confirm this when $s=d$ in~\cref{thm:d=s} and $s=1$ in~\cref{thm:s=1}. In particular,~\cref{thm:Verification} also indicates that~\cref{conj:TrueGeneralization} holds when $d=2$ and $n$ is sufficiently large.

\section{Useful Tools}
The Erd\H{o}s-Rado sunflower lemma~\cite{1960Sunflower} is a classical result in combinatorics, which states that any sufficiently large set system must contain a sunflower. A \emph{sunflower} (or delta-system) is a collection of \( k \) distinct sets \( S_1, S_2, \dots, S_k \) such that their pairwise intersections are identical, called the \emph{core}. Their non-overlapping parts are the \emph{petals}. The lemma gives an upper bound on the size of a set family that avoids a sunflower with \( k \) sets: \(|\mathcal{F}| \leq (k-1)^r r!\), where \( r \) is the maximum set size. In a recent breakthrough, Alweiss, Lovett, Wu, and Zhang~\cite{2021Annals} improved the bound in the sunflower lemma significantly, and the current best-known bound was proven by Bell, Chueluecha, and Warnke~\cite{2021DMSunflower}.

In recent years, significant progress has been made on several important combinatorial problems under the additional assumption of a bounded VC-dimension, such as the Erd\H{o}s-Schur conjecture~\cite{2021ErdosSchur}, the Erd\H{o}s-Hajnal conjecture~\cite{2019DCGBVCEH,2023BVCErdosHajnal} and the sunflower conjecture~\cite{2024BVCSunflower,2023BVCSunflower}. In particular, assuming bounded VC-dimension, Balogh, Bernshteyn, Delcourt, Ferber, and Pham~\cite{2024BVCSunflower} quantitatively strengthened the sunflower lemma. The iterated logarithm of \( k \), denoted as \( \log^\ast k \), is the number of times the logarithm function must be applied iteratively to \( k \) until the result is less than or equal to 1. 

\begin{theorem}[\cite{2024BVCSunflower}]\label{thm:BVCSunflower}
  Let $n\ge k>d$ and $r\ge 1$ be positive integers. There exists a constant $C_{r, d} > 0$ such that  any $k$-uniform set system $\mathcal{F}\subseteq\binom{[n]}{k}$ with VC dimension at most $d$ and $|\mathcal{F}|\ge (C_{r,d}\cdot \log^\ast k)^{k}$ contains a sunflower with $r$ petals.
\end{theorem}

For a set system \(\mathcal{F} \subseteq \binom{[n]}{d+1}\), and for an integer \(1 \leq s \leq d\), we denote the \emph{$s$-shadow} of \(\mathcal{F}\) as 
\[
\partial_s \mathcal{F} := \{T \in \binom{[n]}{s} : T \subseteq F \text{ for some } F \in \mathcal{F}\}.
\]
One of the earliest results in extremal set theory, the celebrated Kruskal-Katona theorem~\cite{1966Katona, 1963Kruskal}, provides a concise and elegant description of the relationship between the cardinalities of \(\mathcal{F} \subseteq \binom{[n]}{d+1}\) and its \(d\)-shadow. Here, we use the Lov\'{a}sz version from~\cite[Exercise 13.31(b)]{1979KKLovasz}, which reformulates the theorem to give a continuous relaxation of the combinatorial result.

\begin{theorem}[Kruskal-Katona theorem]\label{thm:KK}
For any \(\alpha \geq d+1\), if \(\mathcal{G}\) is a \((d+1)\)-uniform set family with \(\abs{\mathcal{G}} = \binom{\alpha}{d+1}\), then 
\[
\abs{\partial_d \mathcal{G}} \geq \binom{\alpha}{d} = \abs{\mathcal{G}} \cdot \frac{d+1}{\alpha - d}.
\]
\end{theorem}

F\"{u}redi and Griggs~\cite{FG86} provided the following characterization of extremal families of shadows, which plays a key role in identifying certain extremal families.

\begin{theorem}\label{thm:FG86}
Let \(n >d \ge g\) be positive integers. Let \(\mathcal{F} \subseteq \binom{[n]}{d+1}\) be a set system such that its \(g\)-shadow is minimal. Then its \((g-1)\)-shadow is also minimal.
\end{theorem}
   We will also take advantage of a recent result in~\cite{2023partialShadow} on the partial shadows.
    \begin{theorem}[\cite{2023partialShadow}]\label{thm:PartialShadow}
       Let $n\ge d+1$ and $k$ be positive integers. Let $\mathcal{F}\subseteq\binom{n}{d+1}$ and $\mathcal{G}\subseteq\binom{[n]}{d}$ be families such that for any $F\in\mathcal{F}$, there exist $G_{1},\ldots,G_{k}\in\mathcal{G}$ such that $G_{i}\subseteq F$ for any $i\in [k]$. Then if $|\mathcal{F}|=\binom{x}{k}$ for some real number $x\ge k$, we have $|\mathcal{G}|\ge\binom{x}{k-1}$.  
    \end{theorem}
The only extremal family in the Erd\H{o}s-Ko-Rado theorem~\cite{1961EKR} is the \emph{star} family, also referred to as a \emph{trivially intersecting} family. Hilton and Milner~\cite{1967Hilton} proved a stability result, demonstrating that the size of non-trivially intersecting families is considerably smaller than the maximum size given by the Erd\H{o}s-Ko-Rado theorem.
\begin{theorem}[Hilton-Milner~\cite{1967Hilton}]\label{thm:HiltonMilner}
Let $n,k$ be positive integers with $n>2k$. If $\mathcal{F}\subseteq \binom{[n]}{k}$ is an intersecting family and $\bigcap_{F\in\mathcal{F}}F=\emptyset$, then we have
\begin{equation*}
  |\mathcal{F}| \le \binom{n-1}{k-1} - \binom{n-k-1}{k-1}+1.   
\end{equation*}
\end{theorem}

\section{Proof of~\cref{thm:GeneralD}}
\label{sec:main-proof}
Throughout the section, let $d\geq 2$ be a fixed integer and $n$ be a sufficiently large integer compared to $d$. All the constants in the proof below are positive constants depending only on $d$, and we will not always state their explicit dependence on $d$. 

Let $\mathcal{F}=\{F_{1},\ldots,F_{m}\}\subseteq \binom{[n]}{d+1}$ be a set system with VC-dimension at most $d$. We can assume that $m\ge\binom{n-1}{d}$, otherwise we are already done. By the definition of VC-dimension, for each set $F_{i}\in\mathcal{F}$, $i\in [m]$, there must exist at least one subset $B_{i}\subsetneq F_{i}$ such that $F\cap F_{i}\neq B_{i}$ for any $F\in\mathcal{F}$. From this point onward, we always assume that when selecting $B_{i}$ for each $i\in [m]$, we choose a set with the largest cardinality among all available options.

\subsection{Overview of the proof}
Let us first give an overview of our proof strategy. The proof consists of three steps. In the first step, we extend Mubayi and Zhao's ingenious sunflower method to prove a preliminary but important estimate, \cref{claim:Enumerate1}. In the second step, we use a double-counting argument on the link of $\cF$ to prove the key \cref{lemma:NearTransversal}. This lemma constructs a small set $J \subset [n]$ that intersects almost every set in $\cF$. We note that such a $J$ does not exist for the tight example $\binom{[n]}{ \leq d}$ of the non-uniform \cref{lem:SSLemma}, so its existence in the uniform case sheds light on why the Frankl--Pach problem has a different answer. Finally, we split $\cF$ into $6$ subfamilies $\cT_J^1, \ldots, \cT_J^6$ based the set $J$ constructed in the second step. With sophisticated structural analysis, we construct various near-injections (maps that are injective on most of the domain) to control the size of each part (\cref{claim:T1} through \cref{claim:T6}) and complete the proof.

We first prove a simple observation by the definition of $B_{i}$.
\begin{claim}\label{claim:Injective}
Let $\mathcal{G}$ be a subfamily of $\mathcal{F}$, and let $\phi: \mathcal{G} \to 2^{[n]}$ be a map that satisfies $\phi(G) \subseteq G$ for any $G \in \mathcal{G}$. If $\mathcal{G}=\mathcal{G}_{1}\sqcup \mathcal{G}_{2}$, where each $F_{k}\in\mathcal{G}_{1}$ satisfies $\phi(F_{k})=B_{k}$ and $|B_{k}|=d$. Then $\phi$ is injective if $\phi|_{\mathcal{G}_{2}}$ is injective.
\end{claim}
\begin{poc}
    Suppose $\phi|_{\mathcal{G}_{2}}$ is injective, but $\phi$ is not. Then there exist $F_{k_{1}},F_{k_{2}}\in\mathcal{G}$ such that $\phi(F_{k_{1}})=\phi(F_{k_{2}})$. Suppose that $F_{k_{1}}\in\mathcal{G}_{1}$, then we have $B_{k_{1}}=\phi(F_{k_{1}})=\phi(F_{k_{2}})\subseteq F_{k_{2}}$, which yields that $F_{k_{1}}\cap F_{k_{2}}=B_{k_{1}}$, a contradiction. Therefore $F_{k_{1}}\in\mathcal{G}_{2}$, and by symmetry we also have $F_{k_{2}}\in\mathcal{G}_{2}$. However, this is impossible since $\phi|_{\mathcal{G}_{2}}$ is injective. This finishes the proof.
\end{poc}

\subsection{Estimate via sunflower lemma}
The selection criteria for $B_{i}$ allows us to apply the sunflower lemma and derive an upper bound on the number of $F_i$'s sharing the same $B_{i}$. The proof is inspired by and generalizes \cite[Lemma 3]{2007JAC} in Mubayi-Zhao's work.

\begin{claim}\label{claim:Enumerate1}
    For each subset $A\subseteq [n]$, the number of indices $i\in [m]$ with $B_{i}=A$ is at most $C_{\eqref{thm:BVCSunflower}}$, where $C_{\eqref{thm:BVCSunflower}}$ is a constant arising from \cref{thm:BVCSunflower}. Consequently, for each integer $0\le s\le d$, the number of indices $i\in [m]$ with $|B_{i}|=s$ is at most $C_{\eqref{thm:BVCSunflower}}\cdot\binom{n}{s}$.
\end{claim}
\begin{poc}
    By~\cref{thm:BVCSunflower}, it suffices to show that for each subset $A\subseteq [n]$, the set system $\mathcal{F}_{A}:=\{F_{i}\in\mathcal{F}:B_{i}=A\}$ does not contain a sunflower of size $d+3$. 

    Suppose there exists a sunflower with sets $G_{1},G_{2},\ldots,G_{d+3}\in\mathcal{F}_{A}$ with $\bigcap_{j=1}^{d+3}G_{j}=C$. Obviously $|C|\le d$ since $\mathcal{F}_{A}$ is $(d+1)$-uniform, which yields that $|G_{j}\setminus C|>0$ for each $j\in [d+3]$. Moreover, by definition, $A$ must be a proper subset of $C$. By the rule of selection of $B_{i}$'s, we can see for $G_{1}$, there must exist some $F\in\mathcal{F}$ such that 
    \begin{equation*}
        F\cap G_{1}=A\cup (G_{1}\setminus C).
    \end{equation*}
    By definition of sunflower, $G_{2}\setminus C,G_{3}\setminus C,\ldots, G_{d+3}\setminus C$ must be pairwise disjoint, therefore, there exists some $k\in\{2,3,\ldots,d+3\}$ such that $F\cap (G_{k}\setminus C)=\emptyset$. However, this implies that the intersection $F\cap G_{k}=F\cap C=(F\cap G_{1})\cap C=A$, a contradiction. This finishes the proof.
\end{poc}

\subsection{The link of the set system and a near-transversal set}
Given \(F_1, \ldots, F_m\) and their corresponding subsets \(B_1, \ldots, B_m\), we now turn our attention to the \emph{link} of \(\mathcal{F}\). Specifically, for each element \(v \in [n]\), define  
\[
\mathcal{X}_v := \{F_k \setminus \{v\} : v \in B_k, k \in [m]\} \quad \text{and} \quad \mathcal{Y}_v := \{F_k \setminus \{v\} : v \in F_k \setminus B_k, k \in [m]\}.
\]
Then we can view $\mathcal{X}_{v}$ and $\mathcal{Y}_{v}$ as $d$-uniform set systems on ground set $[n]\setminus\{v\}$ and $\mathcal{X}_{v}\cup\mathcal{Y}_{v}$ as the link of $\mathcal{F}$ at $v$.
Clearly, $\mathcal{X}_{v}\cap \mathcal{Y}_{v}=\emptyset$. Moreover, there is a bijection $\mathcal{X}_v \sqcup \mathcal{Y}_v \to \{F_k \in \mathcal{F}: v \in F_k\}$ defined by $S \to S \cup \{v\}$, so each set \( F_k \setminus \{v\} \) in \( \mathcal{X}_v \) or \( \mathcal{Y}_v \) corresponds uniquely to a set \( F_k \) in \( \mathcal{F} \), and thus to a corresponding \( B_k \).

Next, for a fixed \(v \in [n]\), we proceed to estimate the sizes of \(\mathcal{X}_{v}\) and \(\mathcal{Y}_{v}\) respectively.

\begin{claim}\label{claim:XvSize}
For each $v\in [n]$, the VC-dimension of $\mathcal{X}_{v}$ is at most $d-1$. Consequently, $|\mathcal{X}_{v}| \leq \binom{n-1}{d-1}$.
\end{claim}
\begin{poc}
    For each $X_{k}=F_{k}\setminus\{v\}\in\mathcal{X}_{v}$, we have $X_{k}\cap X_{\ell}\neq B_{k}\setminus\{v\}$ for any other set \linebreak$X_{\ell}:=F_{\ell}\setminus\{v\}\in\mathcal{X}_{v}$, otherwise we have $F_{k}\cap F_{\ell}=B_{k}$, a contradiction. Therefore $\mathcal{X}_{v}\subseteq\binom{[n]\setminus\{v\}}{d}$ has VC-dimension at most $d-1$, which yields $|\mathcal{X}_{v}| \leq\binom{n-1}{d-1}$ by \cref{thm:FranklPach}.
\end{poc}

\begin{claim}\label{YvSize}
   For each \(v \in [n]\), we have \(|\mathcal{Y}_{v}| \le C_{1} \left( \binom{n-1}{d-1} - |\mathcal{X}_{v}| \right)^{\frac{d}{d-1}} + C_{2}n^{\frac{d(d-2)}{d-1}}\) for some constants $C_{1},C_{2}>0$ which only depend on $d$.
\end{claim}
\begin{poc}
    For each $v\in [n]$, let 
    \[\mathcal{B}_{v}:=\{B_{k}\setminus\{v\}:v\in B_{k}, |B_{k}|=d\}.
    \]
    By definition of $X_{v}$, we can see
    \[
    |\mathcal{B}_{v}|=|\mathcal{X}_{v}|-|\{k\in [m]:v\in B_{k},|B_{k}|\le d-1\}|\ge |\mathcal{X}_{v}|-c_{1}n^{d-2},
    \]
    for some constant $c_{1}>0$, where the last inequality holds by~\cref{claim:Enumerate1}.
    On the other hand, by definition of $\mathcal{Y}_{v}$, any set $Y\in\mathcal{Y}_{v}$ cannot contain any set of $\mathcal{B}_{v}$. In other words, we have
    \[
    \partial_{d-1}\mathcal{Y}_{v}\cap\mathcal{B}_{v}=\emptyset.
    \]
    Then by~\cref{thm:KK}, we have
    \[
|\mathcal{Y}_{v}| \le c_{2} |\partial_{d-1} \mathcal{Y}_{v}|^{\frac{d}{d-1}} \le c_{2} \left( \binom{n-1}{d-1} - |\mathcal{B}_{v}| \right)^{\frac{d}{d-1}} \le c_{2} \left( \binom{n-1}{d-1} - |\mathcal{X}_{v}| + c_{1}n^{d-2} \right)^{\frac{d}{d-1}},
\]
for some constant $c_{2}>0$. Moreover, since $(x+y)^{\frac{d}{d-1}}\le 2(x^{\frac{d}{d-1}}+y^{\frac{d}{d-1}})$, we have
\[
|\mathcal{Y}_{v}| \le C_{1} \left( \binom{n-1}{d-1} - |\mathcal{X}_{v}| \right)^{\frac{d}{d-1}} + C_{2} n^{\frac{d(d-2)}{d-1}}.
\]
for some $C_{1},C_{2}$ which only depend on $d$. This finishes the proof.
\end{poc}

We now present one of the critical components of our proof: we can always identify a transversal set with favorable properties.

\begin{lemma}\label{lemma:NearTransversal}
   Let \(d \geq 2\) be an integer and \(n\) be sufficiently large compared to $d$. If the VC-dimension of \(\mathcal{F} \subseteq \binom{[n]}{d+1}\) is at most \(d\), then for any integer \(1 \leq s \leq n\), there exists a set \(J\) of size at most \(c_{3} \cdot s\) such that  
\[
\sum_{v \in J} |\mathcal{Y}_{v}| \geq |\mathcal{F}| - C_{3} \left( s^{-\frac{1}{d-1}} n^{d} + n^{\frac{d^{2} - d - 1}{d-1}} \right),
\]
where $c_{3},C_{3}>0$ are constants only depending on $d$.
\end{lemma}
\begin{proof}
    We first apply double counting based on~\cref{claim:XvSize} and~\cref{YvSize}. By definitions, we first observe that  
\[
\sum_{v \in [n]} |\mathcal{X}_v| = \sum_{k \in [m]} |B_k|.
\]  
Notice that 
\[
\sum_{k \in [m]} |B_k| = \sum_{j=0}^{d} j \cdot |\{k \in [m] : |B_k| = j\}| \geq d \cdot |\mathcal{F}| - d \cdot |\{k \in [m] : |B_k| \leq d-1\}|.
\]  
By~\cref{claim:Enumerate1} and our assumption $\abs{\cF} \geq \binom{n - 1}{d}$, we have  
\[
\sum_{v \in [n]} |\mathcal{X}_{v}| \geq d \binom{n-1}{d} - c_{4} n^{d-1},
\]  
for some constant \(c_{4} > 0\), which further yields that
\begin{equation}\label{equation:SumXv}
    \sum_{v \in [n]} \left( \binom{n-1}{d-1} - |\mathcal{X}_{v}| \right) \leq n\binom{n-1}{d-1} - d\binom{n-1}{d} + c_{5}n^{d-1} \leq c_{5}n^{d-1},
\end{equation}
where $c_{5}=c_{4}\cdot d$. Note that by~\cref{claim:XvSize}, every term on the left-hand side is non-negative. On the other hand, since $\mathcal{F}$ is $(d+1)$-uniform, we have
\[
\sum\limits_{v\in [n]}(|\mathcal{X}_{v}|+|\mathcal{Y}_{v}|)=(d+1)\cdot |\mathcal{F}|.
\]
This leads to
\begin{equation}\label{equation:YvSum}
    \sum\limits_{v\in [n]}|\mathcal{Y}_{v}|=(d+1)|\mathcal{F}|-\sum\limits_{v\in [n]}|\mathcal{X}_{v}|=\sum\limits_{k\in [m]}(|F_{k}|-|B_{k}|)\ge |\mathcal{F}|.
\end{equation}
Now we take
\[J=\left\{v\in [n]:\binom{n-1}{d-1}-|\mathcal{X}_{v}|\ge \frac{n^{d-1}}{s}\right\}.\]
Then by inequality~\eqref{equation:SumXv}, we have $|J|\le c_{3} s$ for some constant $c_{3}>0$ which only depends on $d$. On the other hand, for each $v\in [n]\setminus J$,~\cref{YvSize} gives that
\[|\mathcal{Y}_{v}|\le C_{1}\left(\binom{n-1}{d-1}-|\mathcal{X}_{v}|\right)^{\frac{d}{d-1}}+C_{2}n^{\frac{d(d-2)}{d-1}}\le C_{1}s^{-\frac{1}{d-1}}\cdot n\left(\binom{n-1}{d-1}-|\mathcal{X}_{v}|\right)+C_{2}n^{\frac{d(d-2)}{d-1}}.\]
By inequality~\eqref{equation:SumXv}, we have
\[
\sum\limits_{v\in [n]\setminus J}|\mathcal{Y}_{v}|\le C_{1}'s^{-\frac{1}{d-1}}n^{d}+C_{2}'n^{\frac{d^{2}-d-1}{d-1}},
\]
for some constant $C_{1}',C_{2}'>0$ which only depend on $d$. Finally, by inequality~\eqref{equation:YvSum}, we have
\[\sum\limits_{v\in J}|\mathcal{Y}_{v}|\ge |\mathcal{F}|-C_{3}\left(s^{-\frac{1}{d-1}}n^{d}+n^{\frac{d^{2}-d-1}{d-1}}\right),\]
for some constant $C_{3}>0$. This finishes the proof.
\end{proof}
\subsection{Controlling the partition via near-injections}
Let $s\leq \frac{n}{4c_3}$ be an integer, which we will set later, and $J\subseteq [n]$ be the subset satisfying the conditions in~\cref{lemma:NearTransversal}. We partition $\mathcal{F}$ into six parts based on the chosen $J$: \(\mathcal{F}:=\bigcup_{i=1}^{6}\mathcal{T}_{J}^{i},\) where
\begin{itemize}
    \item $\mathcal{T}_{J}^{1}$ consists of the sets $F_{k}\in\mathcal{F}$ with $F_{k}\in\binom{[n]\setminus J}{d+1}$ and $B_{k}\in\binom{[n]\setminus J}{d-1}$.
    \item $\mathcal{T}_{J}^{2}$ consists of the sets $F_{k}\in\mathcal{F}$ with $B_{k}\in\binom{[n]\setminus J}{d}$.
    \item $\mathcal{T}_{J}^{3}$ consists of the sets $F_{k}\in\mathcal{F}$ with $|F_{k}\cap J|=1$ and $B_{k}\in\binom{[n]\setminus J}{d-1}$.
    \item $\mathcal{T}_{J}^{4}$ consists of the sets $F_{k}\in\mathcal{F}$ with $|F_{k}\cap J|=1$ and $B_{k}\in J\times \binom{[n]\setminus J}{d-1}$.
    \item $\mathcal{T}_{J}^{5}$ consists of the sets $F_{k}\in\mathcal{F}$ with $|F_{k}\cap J|=2$ (i.e., $F_{k}\in\binom{J}{2}\times\binom{[n]\setminus J}{d-1}$) and $F_{k}\cap ([n]\setminus J)\subseteq B_{k}$.
    \item $\mathcal{T}_{J}^{6}=\mathcal{F}\setminus (\bigcup\limits_{i=1}^{5}\mathcal{T}_{J}^{i})$.   
\end{itemize}
As an example and motivation for our subsequent arguments, we compute these families explicitly for Mubayi and Zhao's construction in \cref{sec:conclusion}. 

Now we define a $d$-uniform hypergraph $\mathcal{G}_{J}$ with the vertex set $[n]\setminus J$ and the edge set $$E(\mathcal{G}_{J}):=\{F_{k}\setminus J:|F_{k}\cap J|=1,F_{k}\in\mathcal{F}\}.$$ 
Assume that $|\mathcal{F}|\ge\binom{n-1}{d}$. By the selection of \( J \) and~\cref{lemma:NearTransversal}, we have $|J|\le c_{3}\cdot s$ and
\[
\sum_{v \in J} |\mathcal{Y}_{v}| = |\mathcal{F}| - C_{3} \left( s^{-\frac{1}{d-1}} n^{d} + n^{\frac{d^{2} - d - 1}{d-1}} \right).
\]  
By definition we have \[E(\mathcal{G}_{J}) \supseteq \bigcup_{v \in J} \bigg(\mathcal{Y}_{v} \cap \binom{[n] \setminus J}{d}\bigg).\] Moreover, we define
\[\mathcal{Y}_{v}':=\{F_{k}\setminus\{v\}:v\in F_{k}\setminus B_{k},|B_{k}|=d\}.\]
Observe that
$$\mathcal{Y}_{v} = \mathcal{Y}_{v}' \cup \{F_{k}\setminus\{v\}:v\in F_{k}\setminus B_{k},|B_{k}| \leq d - 1\}.$$
By~\cref{claim:Enumerate1}, we have $|\mathcal{Y}_{v}'|\ge |\mathcal{Y}_{v}|-C_{\eqref{thm:BVCSunflower}}\sum\limits_{j=0}^{d-1}\binom{n}{j}$. By the definition of $B_k$, the sets in $\mathcal{Y}'_v$ are pairwise disjoint. Thus, we can obtain a lower bound for the number of edges in $\mathcal{G}_{J}$.

\begin{equation}\label{equ:AboutGJ}
\begin{aligned}
    |E(\mathcal{G}_{J})| &\geq \sum_{v \in J} \bigg|\mathcal{Y}_{v}' \cap \binom{[n] \setminus J}{d}\bigg| \\
    &\geq \sum_{v \in J} \Bigg(|\mathcal{Y}_{v}| - \bigg(\binom{n}{d} - \binom{n-|J|}{d}\bigg)-C_{\eqref{thm:BVCSunflower}}\sum_{j=0}^{d-1}\binom{n}{j}\Bigg) \\
    &\geq \sum_{v \in J} \big(|\mathcal{Y}_{v}| - c_{6}|J|n^{d-1}\big) \\
    &\geq |\mathcal{F}| - C_{4}\big(s^{-\frac{1}{d-1}}n^{d} + n^{\frac{d^{2}-d-1}{d-1}} + |J|^{2}n^{d-1}\big) \\
    &\geq \binom{n-|J|}{d} - C_{4}\big(s^{-\frac{1}{d-1}}n^{d} + n^{\frac{d^{2}-d-1}{d-1}} + s^{2}n^{d-1}\big),
\end{aligned}
\end{equation}
where $C_{4},c_{6}>0$ are constants which only depend on $d$ and the fourth inequality follows from~\cref{lemma:NearTransversal}. By inequality~\eqref{equ:AboutGJ} and straightforward calculations, we obtain  
\[
|E(\mathcal{G}_{J})| \geq \binom{n-|J|}{d} - C_{4}\big(s^{-\frac{1}{d-1}}n^{d} + n^{\frac{d^{2}-d-1}{d-1}} + s^{2}n^{d-1}\big)= \binom{n - |J|}{d} - Kn,
\]  
where \[ K:=K(s) = C_{4}\left(s^{-\frac{1}{d-1}}n^{d-1}+n^{\frac{d^{2}-2d}{d-1}}+s^{2}n^{d-2}\right).\]

We proceed to carefully analyze the size of each subfamily \( \mathcal{T}_{J}^{i} \) for \( i \in [6] \) and establish a series of claims. Before doing so, we present a simple fact, derived from the denseness of the hypergraph \( \mathcal{G}_{J} \), that will be repeatedly used in the following proofs.

\begin{fact}\label{fact:DenseHypergraph}
For $|J|\le \frac{n}{4}$ (guaranteed by the assumption $s\leq \frac{n}{4c_3}$), let $\mathcal{S}\subseteq\binom{[n]\setminus J}{d-1}$ consist of $(d-1)$-subsets which are contained in at most $100d+|J|$ edges of $\mathcal{G}_{J}$, then $|\mathcal{S}|\le C_{5}K$ for some constant $C_{5}>0$.
\end{fact}

\begin{claim}\label{claim:T1}
    $|\mathcal{T}_{J}^{1}|\le C_{6} K$ for some constant $C_{6}>0$.
\end{claim}
\begin{poc}
Recall that \( \mathcal{T}_{J}^{1} \) consists of the sets \( F_{k} \in \mathcal{F} \) such that \( F_{k} \subseteq \binom{[n] \setminus J}{d+1} \) and \( B_{k} \in \binom{[n] \setminus J}{d-1} \). Let \( F_{k} \in \mathcal{T}_{J}^{1} \) with \( B_{k} \subseteq \binom{[n] \setminus J}{d-1} \). If \( \boldsymbol{e} \in E(\mathcal{G}_{J}) \subseteq \binom{[n] \setminus J}{d} \) is an edge containing \( B_{k} \), then by definition, there exists some vertex \( v \in J \) such that \( F_{\ell} = \boldsymbol{e} \cup \{v\} \in \mathcal{F} \). Since \( F_{k} \cap F_{\ell} \neq B_{k} \), we have \( \boldsymbol{e} \subseteq F_{k} \). Therefore, there are at most \( d+1 \) edges in \( \mathcal{G}_{J} \) containing \( B_{k} \). 

Moreover, by \cref{claim:Enumerate1}, each such \( B_{k} \) corresponds to at most \( C_{\eqref{thm:BVCSunflower}} \) distinct sets in \( \mathcal{F} \). Combining this with \cref{fact:DenseHypergraph}, we obtain \( |\mathcal{T}_{J}^{1}| \leq C_{6} K \), where \( C_{6}=C_{5}\cdot C_{\eqref{thm:BVCSunflower}}.\) This finishes the proof. 
\end{poc}

\begin{claim}\label{claim:T2T3}    $|\mathcal{T}_{J}^{2}\cup\mathcal{T}_{J}^{3}|\le\binom{n-|J|}{d}+C_{7}(K+\sqrt{Kn^{d-1}}+n^{d-2})$ for some constant $C_{7}>0$.
\end{claim}
\begin{poc}
   Recall that $\mathcal{T}_{J}^{2}$ consists of the sets $F_{k}\in\mathcal{F}$ with $B_{k}\in\binom{[n]\setminus J}{d}$ and $\mathcal{T}_{J}^{3}$ consists of the sets $F_{k}\in\mathcal{F}$ with $|F_{k}\cap J|=1$ and $B_{k}\in\binom{[n]\setminus J}{d-1}$. We define a map $\phi:\mathcal{T}_{J}^{2}\cup\mathcal{T}_{J}^{3}\rightarrow\binom{[n]\setminus J}{d}$ as 
\[
\phi(F_k) =
\begin{cases}
B_k, & \text{if } F_k \in \mathcal{T}_{J}^2, \\
F_k \setminus J, & \text{if } F_k \in \mathcal{T}_{J}^3.
\end{cases}
\]
By definition of $\phi$, we have
\begin{equation*}
|\mathcal{T}_{J}^{2}\cup\mathcal{T}_{J}^{3}|=\sum\limits_{B\in\binom{[n]\setminus J}{d}}|\phi^{-1}(B)|\le \binom{n-|J|}{d}+\sum\limits_{B\in\binom{[n]\setminus J}{d},|\phi^{-1}(B)|\ge 2}|\phi^{-1}(B)|.
\end{equation*}
Define \[ \mathcal{B}_{L} := \{B \in \binom{[n] \setminus J}{d} : |\phi^{-1}(B)| \geq 2\}. \]  
For any \( B \in \mathcal{B}_{L} \), if \( \phi^{-1}(B) \) contains a set \( F_{k} \in \mathcal{T}_{J}^{2} \), then by the definition of \( B_{k} \), no other set in \( \mathcal{F} \) can contain \( B_{k} \), implying \( |\phi^{-1}(B)| = 1 \), which is a contradiction. Hence, for any \( B \in \mathcal{B}_{L} \), \( \phi^{-1}(B) \) can only contain sets from \( \mathcal{T}_{J}^{3} \).

For each \( B \in \mathcal{B}_{L} \), if \( F_{k} \in \phi^{-1}(B) \subseteq \mathcal{T}_{J}^{3} \), then \( B_{k} \subseteq B \). Consequently, the number of such \( B_{k} \) is at most \( d \). Furthermore, by \cref{claim:Enumerate1}, each such \( B_{k} \) corresponds to at most \( C_{\eqref{thm:BVCSunflower}} \) distinct sets in \( \mathcal{F} \). Thus, we have \( |\phi^{-1}(B)| \leq c_{7} \), where \( c_{7}=C_{\eqref{thm:BVCSunflower}}d \). Combining these observations, we conclude that  
\[
|\mathcal{T}_{J}^{2} \cup \mathcal{T}_{J}^{3}| \leq \binom{n - |J|}{d} + c_{7} \cdot |\mathcal{B}_{L}|.
\]

Next, it remains to derive an upper bound for \( |\mathcal{B}_{L}| \). Let \( F_{k} = B \cup \{j\} \) and \( F_{\ell} = B \cup \{j'\} \) be two sets in \( \phi^{-1}(B) \), where $j,j'\in J$. We then consider the following two cases.

\begin{Case}
\item Suppose \( B_{\ell} = B_{k} \). In this case, \( B \) is the only edge in the hypergraph \( \mathcal{G}_{J} \) that contains \( B_{k} = B_{\ell} \); otherwise, there would exist some \( F \in \mathcal{F} \) such that \( F \cap B = B_k = B_\ell \) and \( \abs{F \cap J} = 1 \). Thus we have either \( F \cap F_{k} = B_{k} \) or \( F \cap F_{\ell} = B_{\ell} \), which leads to a contradiction. Therefore, by \cref{fact:DenseHypergraph}, the number of such \( B_{k} \) is at most \( C_{5}K \). 

Moreover, since \( B \) is an edge in \( \mathcal{G}_{J} \) by definition, it is uniquely determined by \( B_{k} = B_{\ell} \). Thus, we conclude that the number of such \( B \in \mathcal{B}_L \) is also at most \( C_{5}K \).

\item Suppose that \( B_\ell \neq B_k \). Let \( E = B_k \cap B_\ell \). Since \( |B_k| = |B_\ell| = d-1 \), it follows that \( |E| = d-2 \). We denote \( B_k = E \cup \{v\} \) and \( B_\ell = E \cup \{w\} \) for some elements \( v, w \in [n] \). We refer to \( E \) as the \emph{core} of \( B \), and we say that the 4-tuple \( (v, w, j, j') \) \emph{corresponds to} this \( B \).

Observe that the set $B=E\cup\{v,w\}$ is uniquely determined by the core $E$ and the corresponding $4$-tuple. For a fixed $E\in\binom{[n]}{d-2}$, we define $N(E)$ to be the number of $B$'s in this case with respect to the core $E$. 

Now let $B'$ be any other set that falls into this case with respect to the same core $E$. First, we claim that $B'\cap B=E$. Suppose not, then assume that $B'=E\cup\{v,w'\}$ for some $w'\neq w\in [n]$ with the corresponding $4$-tuple being $(v,w',h,h')$ for some $h,h'\in J$. Then both of $F_{k}=E\cup\{v,w,j\}$ and $F_{h}=E\cup\{v,w',h\}$ belong to $\mathcal{F}$, therefore we have 
\[(E\cup\{v,w,j\})\cap (E\cup\{v,w',h\})\neq B_{k}=E\cup\{v\},\]
which yields that $j=h$. Similarly, we can obtain that
\[(E\cup\{v,w,j\})\cap (E\cup\{v,w',h'\})\neq B_{k}=E\cup\{v\},\]
which leads to $j=h'$, a contradiction.

We can then assume \( B' = E \cup \{v', w'\} \), where \( \{v', w'\} \cap \{v, w\} = \emptyset \), with the corresponding 4-tuple \( (v', w', h, h') \). Our next goal is to show that there exists at least one pair \( M \subseteq \{v, w\} \times \{v', w'\} \) such that \( M \cup E \notin E(\mathcal{G}_J) \).

Without loss of generality, we may assume \( j \neq h \). We claim that \( \{v, v'\} \cup E \) is not an edge in the hypergraph \( \mathcal{G}_J \). Suppose otherwise, then there exists some \( F_q = \{v, v', q\} \cup E \in \mathcal{F} \) with \( q \in J \). Since 
\[
(\{v, v', q\} \cup E) \cap (\{v, w, j\} \cup E) \neq B_k = E \cup \{v\},
\]
we deduce \( q = j \). Similarly, we have 
\[
(E \cup \{v, v', q\}) \cap (E \cup \{v', w', h\}) \neq E \cup \{v'\},
\]
which implies \( q = h \), leading to a contradiction.

If $(v_{1},w_{1},*,*),\ldots,(v_{t},w_{t},*,*)$ are the $4$-tuples corresponding to the $B$'s that fall into \textbf{Case 2} with respect to the core $E$, then $\{v_{1},w_{1}\},\ldots,\{v_{t},w_{t}\}$ are pairwise disjoint. Moreover, for each pair $\{i,j\}\in\binom{[t]}{2}$, there is at least one pair $M\in\{v_{i},w_{i}\}\times\{v_{j},w_{j}\}$ such that $M\cup E\in\binom{[n]\setminus J}{d}\setminus E(\mathcal{G}_{J})$. Letting $N(E) = t$, we then have
\[\binom{N(E)}{2}\le |\{A\in\binom{[n]\setminus J}{d}\setminus E(\mathcal{G}_{J}):E\subseteq A\}|.\]
Summing over all possible $(d-2)$-subsets $E$, we then have
\[\sum\limits_{E\in\binom{[n]\setminus J}{d-2}}\binom{N(E)}{2}\le c_{8}Kn,\]
for some constant $c_{8}>0$ which only depends on $d$. Therefore we have
\[\sum\limits_{E\in\binom{[n]\setminus J}{d-2}}N(E)^{2}\le\sum\limits_{E\in\binom{[n]\setminus J}{d-2}}\left(4\binom{N(E)}{2}+1\right)\le 4c_{8}Kn+n^{d-2},\]
where we take advantage of $x^{2} \leq 4\binom{x}{2}+1$. Then by the Cauchy-Schwarz inequality, we have
\[\sum\limits_{E\in\binom{[n]\setminus J}{d-2}}N(E)\le c_{9}\left(\sqrt{Kn^{d-1}}+n^{d-2}\right),\]
for some constant $c_{9}>0$ which only depends on $d$. Therefore, the total number of distinct $B$'s falling into this case is at most $c_{9}\left(\sqrt{Kn^{d-1}}+n^{d-2}\right)$.
\end{Case}
Combining these cases, we conclude that \[|\mathcal{B}_{L}|\le C_{7}\left(K+\sqrt{Kn^{d-1}}+n^{d-2}\right),\]
where $C_{7}=c_{7}\cdot \max\{c_{9},C_{5}\}$. This finishes the proof.
\end{poc}

\begin{claim}\label{claim:T4T5}
    $|\mathcal{T}_{J}^{4}\cup\mathcal{T}_{J}^{5}|\le (|J|-1)\binom{n-|J|}{d-1}+2C_{5}K$.
\end{claim}
\begin{poc}
    Recall that $\mathcal{T}_{J}^{4}$ consists of the sets $F_{k}\in\mathcal{F}$ with $|F_{k}\cap J|=1$ and $B_{k}\in J\times \binom{[n]\setminus J}{d-1}$. and $\mathcal{T}_{J}^{5}$ consists of the sets $F_{k}\in\mathcal{F}$ with $|F_{k}\cap J|=2$ and $F_{k}\cap ([n]\setminus J)\subseteq B_{k}$. For each $C \in \binom{[n] - J}{d - 1}$, let \[\mathcal{T}_{J,C}^{4}:=\{F_{k}\in\mathcal{T}_{J}^{4}:C\subseteq B_{k}\},\] and
    \[\mathcal{T}_{J,C}^{5}:=\{F_{k}\in\mathcal{T}_{J}^{5}:C\subseteq F_{k}\}.\] By definition, each element of $\cT_J^4 \cup \cT_J^5$ lies in $\cT_{J, C}^4 \cup \cT_{J, C}^5$ for exactly one $C$. So we have
    $$\abs{\cT_J^4 \cup \cT_J^5}  = \sum_{C \in \binom{[n]\setminus J}{d - 1}} \abs{\cT_{J, C}^4 \cup \cT_{J, C}^5}.$$
   For each $C \in \binom{[n]\setminus J}{d - 1}$, define a map $\phi_C: \cT_{J, C}^4 \cup \cT_{J, C}^5 \to J$ as follows.
   \begin{itemize}
       \item For each set $F_k \in \cT_{J, C}^4$ and each set $F_{k} \in \cT_{J, C}^5$ with $\abs{B_k} = d$, \(\phi_{C}\) maps $F_k$ to the unique element in $B_k \setminus C$.
       \item It remains to map the sets \( F_k \) in \( \mathcal{T}_{J, C}^5 \) where \( B_k = C \). We define an auxiliary graph \(G^5_C \) with vertex set \( J \) and edges given by \( \{F_k \cap J : F_k \in \mathcal{T}_{J, C}^5, B_k = C\} \). In this graph, every pair of edges must intersect, meaning that the graph is either empty, a star, or a triangle. In each case, we can define \( \phi_C(F_k) \) such that it maps \( \{F_k : F_k \in \mathcal{T}_{J, C}^5, B_k = C\} \) injectively to a non-isolated vertex in \( G^5_C \), with \( \phi_C(F_k) \subseteq F_k \).
   \end{itemize}

   We apply \cref{claim:Injective} to the map \( \phi_C'(F) = \{\phi_C(F)\} \cup C \). Note that \( \phi_C' \) is injective if \( \phi_C \) is injective on \( \{F_k : F_k \in \mathcal{T}_{J, C}^5, B_k = C\} \), which holds by definition. Consequently, both \( \phi_C' \) and \( \phi_C \) are injective. Thus, we obtain \[
\abs{\mathcal{T}_J^4 \cup \mathcal{T}_J^5} \leq (\abs{J} - 1) \binom{n - \abs{J}}{d - 1} + |\{C : \phi_C \text{ is surjective}\}|.
\]

   If $G_C^5$ is a star, then $\phi_v$ must miss a non-isolated vertex, so it cannot be surjective. Therefore, $\phi_v$ is surjective only if $G_C^5$ is empty or a triangle. Next, we discuss these two cases separately.
\begin{Case}
    \item Suppose \( G_C^5 \) is empty. Then every set \( F_k \) in \( \mathcal{T}_{J, C}^4 \cup \mathcal{T}_{J, C}^5 \) satisfies \( \abs{B_k} = d \), and for each \( j \in J \), there exists some \( F_k \in \mathcal{T}_{J, C}^4 \cup \mathcal{T}_{J, C}^5 \) with \( B_k = C \cup \{j\} \). Now suppose \( \boldsymbol{e} \) is any edge in \( \mathcal{G}_J \) containing \( C \). Then there exists \( j_0 \in J \) such that \( \boldsymbol{e} \cup \{j_0\} \in \mathcal{F} \). Therefore, there exists \( F_k \in \mathcal{T}_{J, C}^4 \cup \mathcal{T}_{J, C}^5 \) with \( B_k = C \cup \{j_0\} \), and this is possible only if \( F_k = \boldsymbol{e} \cup \{j_0\} \). 

Thus, any edge in \( \mathcal{G}_J \) containing \( C \) is of the form \( F_k \cap ([n] \setminus J) \) for some \( F_k \in \mathcal{T}_{J, C}^4 \cup \mathcal{T}_{J, C}^5 \). In particular, there are at most \( \abs{J} \) such edges. Then, by~\cref{fact:DenseHypergraph}, the total number of such \( C \) is at most \( C_{5} K \).
    \item Suppose \( G_C^5 \) is a triangle with vertices \( \{a, b, c\} \). Then \( F_{k_1} = C \cup \{a, b\} \), \( F_{k_2} = C \cup \{b, c\} \), and \( F_{k_3} = C \cup \{a, c\} \) are all sets in \( \mathcal{F} \), where \( B_{k_i} = C \) holds for all \( i \in [3] \). By the definition of \( B_{k_i} \), there can be no other element of \( \mathcal{F} \) containing \( C \). In particular, this implies that no edge in \( \mathcal{G}_J \) can contain \( C \). 

Thus, by~\cref{fact:DenseHypergraph}, the total number of such \( C \) is at most \( C_5 K \).
\end{Case}
We conclude that
\[
|\{C : \phi_C \text{ is surjective}\}| \leq 2C_{5} K,
\]
 which leads to
\[
\abs{\mathcal{T}_J^4 \cup \mathcal{T}_J^5} \leq (\abs{J} - 1) \binom{n - \abs{J}}{d - 1} + \abs{\mathcal{S}} \leq (\abs{J} - 1) \binom{n - \abs{J}}{d - 1} + 2C_5 K.
\]
 This completes the proof.

\end{poc}

\begin{claim}\label{claim:T6}
    $|\mathcal{T}_{J}^{6}|\le C_{8}|J|^{2}n^{d-2}$ for some constant \(C_{8}>0\) which only depends on $d$.
\end{claim}
\begin{poc}
    We first claim that for each $F_{k}\in \mathcal{T}_{J}^{6}$, $|B_{k}\cap ([n]\setminus J)|\le d-2$. Suppose there is some $F_{k}\in\mathcal{T}_{J}^{6}$ such that $|B_{k}\cap ([n]\setminus J)|\ge d-1$. We consider the following cases.
    \begin{Case}
\item If $|F_{k}\cap J|=0$, then $F_{k}$ belongs to either $\mathcal{T}_{J}^{1}$ or $\mathcal{T}_{J}^{2}$, a contradiction.
\item  If $|F_{k}\cap J|=1$, then $F_{k}$ belongs to $\mathcal{T}_{J}^{2}$, $\mathcal{T}_{J}^{3}$ or $\mathcal{T}_{J}^{4}$, a contradiction.
\item If $|F_{k}\cap J|=2$, then $F_{k}$ belongs to $\mathcal{T}_{J}^{5}$, a contradiction.

\item If $|F_{k}\cap J|\geq 3$, then since $B_{k}\subseteq F_{k}$, we can see $|B_{k}\cap (F_{k}\setminus J)|\le d-2$, which yields that $|B_{k}\cap ([n]\setminus J)|\le d-2$, a contradiction to our assumption.
    \end{Case}
We observe that there are at most \( \binom{|J|}{2} \cdot \sum_{i=0}^{d-2} \binom{n}{i} \le c_{10} |J|^2 n^{d-2} \) distinct sets \( B_k \) such that \( |B_k \cap ([n] \setminus J)| \le d-2 \), where $c_{10}>0$ only depends on $d$. By~\cref{claim:Enumerate1}, there are at most \( C_{\eqref{thm:BVCSunflower}} \) sets \( F \in \mathcal{F} \) corresponding to each \( B_k \). Therefore, it follows that 
\[
|\mathcal{T}_J^6| \le C_{8} |J|^2 n^{d-2},
\]
where $C_{8}=c_{10}\cdot C_{\eqref{thm:BVCSunflower}}$.
\end{poc}

Combining \cref{claim:T1}, \cref{claim:T2T3}, \cref{claim:T4T5}, and \cref{claim:T6}, we obtain the following bound on \(|\mathcal{F}|\): 
\[
|\mathcal{F}| \leq \binom{n-|J|}{d}+(|J|-1)\binom{n-|J|}{d-1}+ C_{9}(K+\sqrt{Kn^{d-1}}+n^{d-2}+|J|^{2}n^{d-2}),
\]
where $C_{9}=\max\{2C_{5}+C_{6}+C_{7},C_{8}\}$. Recall that \[ K:=K(s) = C_{4}\left(s^{-\frac{1}{d-1}}n^{d-1}+n^{\frac{d^{2}-2d}{d-1}}+s^{2}n^{d-2}\right),\] and $|J|\le c_{3}\cdot s$, we then get
\[
\abs{\cF} \leq \binom{n - 1}{d} + C_{10}\left(s^2 n^{d - 2}+ s^{-\frac{1}{2(d - 1)}} n^{d - 1} + n^{\frac{2d^2 - 4d + 1}{2(d - 1)}}+ sn^{d - 3/2} + n^{d - 2}\right),
\]
where $C_{10}>0$ only depends on $d$.

Taking \( s = \lfloor n^{\frac{d-1}{2d-1}}\rfloor \), we obtain the bound
\[
|\mathcal{F}| \leq \binom{n-1}{d} + C_{d}\cdot n^{d - 1 - \frac{1}{4d-2}}
\]
for some constant \( C_{d} > 0 \). This finishes the proof of~\cref{thm:GeneralD}.

\section{Polynomial improvement via linear algebra methods}\label{sec:PolynomialMethod}
We further refine the tools from linear algebra methods to show that the Frankl--Pach upper bound can be improved by \(\Omega_d(n^{d-4})\). Although this result is weaker than our main finding, we believe it is of independent interest due to its potential applicability in providing the sharp result.
\begin{theorem}\label{thm:MainPoly}
    Let $d\ge 4$ be a positive integer and $n$ be sufficiently large compared to $d$. If $\mathcal{F}\subseteq\binom{[n]}{d+1}$ is a set system with VC-dimension at most $d$, then $|\mathcal{F}|\le \binom{n}{d}-c_{d}n^{d-4}$ for some constant $c_{d}>0$. 
\end{theorem}

\subsection{Overview of the proof of~\cref{thm:MainPoly}}

We first briefly recap preceding linear algebra methods. Frankl and Pach associated a vector (forming the so-called ``higher-order inclusion matrix") $\boldsymbol{v}_F \in \mathbb{R}^{\binom{[n]}{d}}$ to each $F \in \cF$ and proved that these vectors are linearly independent. By dimension counting, we have $\abs{\cF} \leq \binom{n}{d}$. Mubayi and Zhao instead worked with $\mathbb{F}^{\binom{[n]}{d}}$ for some finite field $\mathbb{F}$. They inserted $\Omega_d(\log n)$ vectors into $\{\boldsymbol{v}_{F}: F \in \cF\}$ and proved that the resulting family is still linearly independent, thereby gaining $\Omega_d(\log n)$ over the Frankl--Pach bound. Their use of a finite field imposes a number-theoretic condition on $d$. Ge, Xu, Yip, Zhang and Zhao~\cite{2024FranklPach} took a different approach, associating carefully designed degree $d$ multilinear polynomials $f_F, g_G \in \mathbb{R}[x_1, \cdots, x_n]$ for each $F \in \cF$ and $G \in \cG = \binom{[n]}{\leq d - 1}$. The Frankl--Pach bound is recovered by showing that the family of polynomials $\{f_F: F \in \cF\} \sqcup \{g_G: G \in \cG\}$ is linearly independent. By inserting a single extra polynomial into the family, the authors in~\cite{2024FranklPach} improved the Frankl--Pach bound by $1$.

In this section, we refine the approach in~\cite{2024FranklPach} by showing that it is possible to $\Omega_d(n^{d - 4})$ polynomials into the family $\{f_F: F \in \cF\} \cup \{g_G: G \in \cG\}$ while preserving linear independence. Our main innovation is a novel sufficient condition for choosing these additional polynomials. Using combinatorial arguments, we then algorithmically construct a large family satisfying the desired condition.
\subsection{An estimate}
Let $\cF = \{F_1, \ldots, F_m\}$ and $B_1, \ldots, B_m$ be as in \cref{sec:main-proof}. In addition to \cref{claim:Enumerate1}, we use a different double-counting approach to obtain a tight estimate for the number of indices $i$ with $|B_{i}|=d$.  

\begin{claim}\label{claim:IndexD}
    The number of indices $i\in [m]$ with $|B_{i}|=d$ is at most $\binom{n-1}{d}$. Moreover, the equality holds if and only if $\{F_i: i \in [m], |B_i|=d\}$ is a star with maximum size.
\end{claim}
\begin{poc}
Let $\mathcal{Y}=\binom{[n]}{d+1}\setminus\mathcal{F}$. Let $P:=\big\{(i,Y):i\in [m],|B_{i}|=d,B_{i}\subseteq Y,Y\in\mathcal{Y} \big\}$. On one hand, since for each $i\in [m]$ with $|B_{i}|=d$, and for each $x\in [n]\setminus F_{i}$, we have $B_{i}\cup\{x\}\notin \mathcal{F}$, which implies that $(i,B_{i})\in P$ if $|B_{i}|=d$. Set $D:=|\{i : |B_i| = d\}|$, we have 
    \[
|P| \geq (n-d-1) \cdot |D|.
\]
    On the other hand, for each $Y\in\mathcal{Y}$, $Y$ has $d+1$  subsets of size $d$. Moreover, by definition for each subset $Y'\subseteq Y$ of size $d$, there is at most one index $i\in [m]$ such that $B_{i}=Y'$. Therefore, we have
 \[
|P| \leq (d+1) \left( \binom{n}{d+1} - |\mathcal{F}| \right).
\]
Observe that for each distinct $F_{i},F_{j}\in\mathcal{F}$ with $|B_{i}|=|B_{j}|=d$, $B_{i}\neq B_{j}$, therefore $|D|\le |\mathcal{F}|$.
Combining the above inequalities, we have
\begin{equation}\label{eq:boundD}
|D|\le \frac{(d+1)\left( \binom{n}{d+1} - |\mathcal{F}| \right)}{n-d-1}\le \frac{(d+1)\left( \binom{n}{d+1} - |D| \right)}{n-d-1},
\end{equation}
which yields that
\begin{equation*}
    |D|\le\frac{d+1}{n}\binom{n}{d+1}=\binom{n-1}{d}.
\end{equation*}

Furthermore, assume that \(\abs{\mathcal{F}} = \binom{n-1}{d}\) and $|B_i|=d$ for all $i\in [m]$. From the above proof, we deduce that for each \(Y \in \mathcal{Y}\) and each subset \(Y' \subseteq Y\) of size \(d\), there exists an index \(i\) such that \(B_i = Y'\). Therefore, we have \[\partial_d \mathcal{Y} = \{B_i : 1 \leq i \leq \abs{\mathcal{F}}\}.\]
Note that \(\abs{\mathcal{Y}} = \binom{n}{d+1} - \binom{n-1}{d} = \binom{n-1}{d+1}\), and \(\abs{\partial_d \mathcal{Y}} = \abs{\mathcal{F}} = \binom{n-1}{d}\). Hence, the \(d\)-shadow of \(\mathcal{Y}\) is minimized by \cref{thm:KK}. Applying \cref{thm:FG86}, we can conclude that the \(k\)-shadow of \(\mathcal{Y}\) is minimized for each \(k \leq d\). In particular, we have \(\abs{\partial_1 \mathcal{Y}} = n - 1\). In other words, \(\mathcal{F}\) is a star centered at \( [n] \setminus \partial_1 \mathcal{Y} \). This completes the proof.
\end{poc}


\subsection{Upper bound via the polynomial method}
We now develop a strategy that allows us to insert more polynomials into the family introduced in \cite{2024FranklPach}.

\begin{claim}\label{claim:LI}
    If there exists a family $\mathcal{Y}:=\{Y_{1},Y_{2},\ldots,Y_{s}\}\subseteq \binom{[n]}{d+1}\setminus\mathcal{F}$, together with $Z_{i}\subseteq Y_{i}$ of size $d$, such that the following three conditions hold:
    \begin{enumerate}
        \item[\textup{(1)}] For any $i\in [s]$ and $j\in [m]$, $Z_{i}\neq B_{j}$. 
        \item[\textup{(2)}] For any distinct $i,j\in [s]$, $|Y_{i}\cap Y_{j}|\le d-1$.
        \item[\textup{(3)}] For any $i,j\in [s]$ with $i<j$, there is no $k\in [m]$ such that both of $F_{k}\cap Y_{i}=Z_{i}$ and $F_{k}\cap Y_{j}=B_{k}$ hold.
    \end{enumerate}
    Then we have $|\mathcal{F}|\le\binom{n}{d}-s$.
\end{claim}
\begin{poc}
We consider the following polynomials, introduced in \cite{2024FranklPach}.

\begin{enumerate}
    \item For each set $F_{i}\in\mathcal{F}$, let $f_{F_{i}}(\boldsymbol{x})=\prod\limits_{j\in B_{i}}x_{j}\cdot\prod\limits_{j\in F_{i}\setminus B_{i}}(x_{j}-1)-\prod\limits_{j\in F_{i}}x_{j}$.
    \item For each set $Y_{i}\in\mathcal{Y}$, let $y_{Y_{i}}(\boldsymbol{x})=-\prod\limits_{j\in Z_{i}}x_{j}$.
    \item Let $\mathcal{H}=\binom{[n]}{\le d-1}$. For each set $H\in\mathcal{H}$, let $h_{H}(\boldsymbol{x})=\sum\limits_{j\notin H}x_{j}\cdot\prod\limits_{j\in H}x_{j}+(|H|-d-1)\prod\limits_{j\in H}x_{j}$. 
    \end{enumerate}
    It then suffices to show that $\{y_{Y_{i}}\}_{Y_{i}\in\mathcal{Y}}$, $\{f_{F_{i}}\}_{F_{i}\in\mathcal{F}}$ and $\{h_{H}\}_{H\in\mathcal{H}}$ are linearly independent, which implies that $s+|\mathcal{F}|+\sum\limits_{i=0}^{d-1}\binom{n}{i}\le\sum\limits_{i=0}^{d}\binom{n}{i}$, therefore, $|\mathcal{F}|\le\binom{n}{d}-s$. Suppose there exist coefficients (that are not all zero) $\boldsymbol{\alpha}=(\alpha_{1},\alpha_{2},\ldots,\alpha_{s})$, $\boldsymbol{\beta}=(\beta_{1},\beta_{2},\ldots,\beta_{m})$, and $\boldsymbol{\gamma}=(\gamma_{1},\gamma_{2},\ldots,\gamma_{w})$, where $w=|\mathcal{H}|=\sum\limits_{i=0}^{d-1}\binom{n}{i}$, such that
    \begin{equation*}
        \sum\limits_{i\in [s]}\alpha_{i}y_{Y_{i}}+ \sum\limits_{j\in [m]}\beta_{j}f_{F_{j}}+\sum\limits_{k\in [w]}\gamma_{k}h_{H_{k}}=0.
    \end{equation*}
    For each $A\subseteq [n]$, we use $v_A \in \{0,1\}^n$ to denote its indicator vector.  We next substitute the indicator vector corresponding to each set in the families \(\mathcal{Y}\), \(\mathcal{F}\), and \(\mathcal{H}\) into the equation above.
    By the properties (1)--(3) and some direct calculations, we can see
        \begin{equation*}
\left(\begin{matrix}
-\boldsymbol{E}_{s} &\TT^\top & \boldsymbol{0} \\
\RR & -\boldsymbol{E}_{m} & \boldsymbol{0} \\
 * & * & \boldsymbol{A}
\end{matrix}\right)
\begin{bmatrix} \boldsymbol{\alpha} \\ \boldsymbol{\beta} \\  \boldsymbol{\gamma} \\ \end{bmatrix} 
= \boldsymbol{0},
\end{equation*}
where $\boldsymbol{E}_{s}$ and $\boldsymbol{E}_{m}$ are the $s\times s$ and $m\times m$ identity matrices respectively, and the $*$'s denote matrices that we do not compute explicitly. The upper left block is $-\boldsymbol{E}_{s}$ since $\abs{Y_i \cap Y_j} \neq Z_i$ for distinct $i$ and $j$ by condition (2). The matrices $\bT$ and $\bR$ are given by
$$\bT_{k, i} = f_{F_k}(\bv_{Y_i}) = (-1)^{d + 1 - \abs{B_k}}\cdot\boldsymbol{1}_{F_k \cap Y_i = B_k},$$
$$\bR_{k, i} = y_{Y_i}(\bv_{F_k}) = (-1)^{d + 1 - \abs{Z_i}}\cdot\boldsymbol{1}_{F_k \cap Y_i = Z_i}.$$
To obtain a contradiction, it suffices to show that the coefficient matrix
$$\left(\begin{matrix}
-\boldsymbol{E}_{s} &\TT^\top & \boldsymbol{0} \\
\RR & -\boldsymbol{E}_{m} & \boldsymbol{0} \\
 * & * & \boldsymbol{A}
\end{matrix}\right)$$
is invertible. Using elementary row operations, we can transform the above matrix into
$$\boldsymbol{M}=\left(\begin{matrix}
    -\boldsymbol{E}_{s}+ \bT^\top \bR & 0 & 0 \\
    \bR & -\boldsymbol{E}_{m} & 0 \\
    * & * & \bA
\end{matrix}\right),$$
where the entries of $\bT^\top \bR$ are given by
$$(\bT^\top \bR)_{i, j} = \sum_{k = 1}^{\abs{\mathcal{F}}} \bT_{k, i} \bR_{k, j} = \sum_{k: F_k \cap Y_i = B_k, F_k \cap Y_j = Z_j} (-1)^{\abs{B_k} - \abs{Z_j}}.$$
By conditions (1) and (3), the rightmost sum is empty when $i\ge j$. So $\bT^\top \bR$ is upper triangular with zero diagonal. Also, note that $A$ is a lower triangular matrix with nonzero diagonal entries by sorting the sets in $\mathcal{H}$ by size in non-decreasing order. Thus each diagonal block of $\boldsymbol{M}$ is invertible. The determinant of $\boldsymbol{M}$ is the product of the determinants of $-\boldsymbol{E}_{s}+\bT^\top \bR$, $-\boldsymbol{E}_{m}$ and $\boldsymbol{A}$, which are all non-zero. Therefore the transformed matrix $\boldsymbol{M}$ is invertible. Thus the original coefficient matrix is also invertible, as desired.
\end{poc}

\subsection{Finding the desired family}
Our final objective is to identify a family \(\mathcal{Y}\) that satisfies the three properties outlined in~\cref{claim:LI} and is as large as possible. We show that it is always possible to construct such a family with size at least \(\Omega_{d}(n^{d-4})\) when $\abs{\mathcal{F}}$ is close to $\binom{n}{d}$. 

Let $\mathcal{D}:=\{B_{i}:|B_{i}|=d\}$ and $t:=\binom{n}{d}-|\mathcal{F}|$. Set $\Gamma_{d}$ to be a large enough constant which only depends on $d$, we define 
\[
\mathcal{D}' := \left\{ B \in \binom{[n]}{d} : \text{there are at least } \Gamma_{d} \text{ indices } k \in [m] \text{ such that } B \subseteq F_{k} \text{ and } |B_{k}| = d - 2 \right\}.
\]
For each set $S\in\binom{[n]}{d+1}$, we define
\[N(S)=\{T\in\binom{[n]}{d+1}:|T\cap S|\ge d\}.\]
Obviously for each $S\in\binom{[n]}{d+1}$, we have $|N(S)|= (n-d-1)(d+1)+1$.

We select the suitable sets $Y_{1},Y_{2},\ldots,Y_{s}\in\binom{[n]}{d+1}$ together with $Z_{1},Z_{2},\ldots,Z_{s}\in\binom{[n]}{d}$ via the following algorithm.
\begin{enumerate}
    \item[Step 1:] Initialize \(\mathcal{Y} = \binom{[n]}{d + 1} \setminus \mathcal{F}\).  
    \item[Step 2:] For each \(i \in [m]\) with \(\abs{B_i} \leq d -3\), remove \(N(F_i)\) from \(\mathcal{Y}\).  
    \item[Step 3:] Repeat the following steps until termination:  
    \begin{enumerate}
        \item \textbf{Termination Check:} If \(\partial_d \mathcal{Y} \subseteq \mathcal{D}\cup\mathcal{D}'\), terminate.  
        \item \textbf{Find a Candidate:} Otherwise, there exists some \(Y \in \mathcal{Y}\) and \(Z \subseteq Y\), such that \(\abs{Z} = d\) and \(Z \notin \mathcal{D}\cup\mathcal{D}'\). Select such a pair \((Y, Z)\) and assign it to \((Y_i, Z_i)\).  
        \item \textbf{Update \(\mathcal{Y}\):} Remove \(N(Y_i)\) from \(\mathcal{Y}\).  
        \item \textbf{Further Removal:} For each \(k \in [m]\) such that \(F_k \cap Y_i = Z_i\), remove from \(\mathcal{Y}\) every \(Y' \in \mathcal{Y}\) satisfying \(Y' \cap F_k = B_k\).  
    \end{enumerate}
\end{enumerate}  
One can check that the sets $Y_1, \ldots, Y_s$ and corresponding $Z_{1},\ldots,Z_{s}$ produced by the above algorithm satisfy all requirements of \cref{claim:LI}. More precisely, condition (1) is guaranteed by step 3(b), condition (2) is guaranteed by step 3(c), and condition (3) is guaranteed by step 3(d). Applying~\cref{claim:LI}, we conclude that $s \leq t$.

Let $\mathcal{Y}^*$ denote the family $\mathcal{Y}$ at the end of the algorithm. We now analyze the evolution of \(\abs{\mathcal{Y}}\) throughout this algorithm and give a lower bound on $\abs{\mathcal{Y}^*}$. Initially, we have  
\[
\abs{\mathcal{Y}} = \binom{n}{d + 1} - \abs{\mathcal{F}}.
\]  

As mentioned above, \(\abs{N(F_k)} = (n-d-1)(d+1)+1=O_{d}(n)\). Furthermore, by~\cref{claim:Enumerate1}, we conclude that that in step 2, at most \(C_{d} \sum\limits_{i=0}^{d-3} \binom{n}{i} \cdot \big((n - d - 1)(d + 1) + 1\big) = O_{d}(n^{d-2})\) sets are removed from \(\mathcal{Y}\).

Next, we focus on step 3. In each iteration of step 3:  
\begin{itemize}
    \item \textbf{Step 3(c)} We remove \((n - d - 1)(d + 1) + 1=O_d(n)\) sets from \(\mathcal{Y}\).  
    \item \textbf{Step 3(d)} Since \(\abs{Z_i} = d\), there are at most \(n\) possible values of \(k\) such that \(F_k \cap Y_i = Z_i\). For each such \(k\), step 2 ensures that \(\abs{B_k} \geq d - 2\). If \(\abs{B_k} \geq d - 1\), there are at most \(n^2\) sets \(Y' \in \mathcal{Y}\) such that \(Y' \cap F_k = B_k\); if \(\abs{B_k} = d - 2\), there are at most \(n^3\) such sets. However, since \(Z_i \notin \mathcal{D}'\), by definition the number of \(k\) satisfying \(F_k \cap Y_i = Z_i\) and \(\abs{B_k} = d - 2\) is at most \(\Gamma_d\). Therefore, step 3(d) removes at most \(n^3 + \Gamma_d n^3 = (\Gamma_d + 1)n^3\) sets.
\end{itemize} 

Combining these, each iteration of step 3 removes at most $2\Gamma_{d}n^{3}$ sets from \(\mathcal{Y}\). When the algorithm terminates, since $|\mathcal{F}|=\binom{n}{d}-t$ and $t\ge s$, we conclude that  
\[
\abs{\mathcal{Y}^*} \geq \binom{n}{d + 1} - \binom{n}{d} - c_{1}n^{d - 2} - 2\Gamma_d n^3 \cdot t,
\]  
where \(s\) denotes the number of iterations in step 3 and $c_{1}$ is a constant which depend on $d$. When the algorithm terminates, we obtain a family \(\mathcal{Y} = \{Y_1, Y_2, \ldots, Y_s\}\).

Let \(\alpha \in \mathbb{R}_{> 0}\) be the unique positive real number such that \(\abs{\mathcal{D} \cup \mathcal{D}'} = \binom{\alpha}{d}\). When the algorithm terminates, we must have \(\partial_d \mathcal{Y^*} \subseteq \mathcal{D}\cup\mathcal{D}'\). Applying~\cref{thm:KK}, we deduce that  
\[
\abs{\mathcal{Y}^*} \leq \binom{\alpha}{d + 1}.
\]  
This gives the inequality:  
\begin{equation}
\label{eq:key-ineq-1}
\binom{\alpha}{d + 1} \geq \binom{n}{d + 1} - \binom{n}{d} - c_{1}n^{d-2} - 2\Gamma_d n^{3}\cdot t.
\end{equation}
On the other hand, from inequality~\eqref{eq:boundD} in the proof of~\cref{claim:IndexD}, we have  
\begin{equation*}
\label{eq:key-ineq-3old}
\abs{\mathcal{D}} \leq \frac{(d + 1) \left(\binom{n}{d + 1} - \binom{n}{d} + t\right)}{n - d - 1}.
\end{equation*}  
By definition of $\mathcal{D}'$ and~\cref{claim:Enumerate1}, we have $\abs{\mathcal{D}'} \leq \frac{C_{d}(d+1)\binom{n}{d-2}}{\Gamma_{d}}.$ 
Thus we obtain
\begin{equation}
\label{eq:key-ineq-3}
\binom{\alpha}{d} \leq \abs{\mathcal{D}} + \abs{\mathcal{D}'} \leq \frac{(d + 1) \left(\binom{n}{d + 1} - \binom{n}{d} + t\right)}{n - d - 1} + \frac{C_{d}(d+1)\binom{n}{d-2}}{\Gamma_{d}}.
\end{equation}  

What remains is to perform an explicit computation by combining inequalities~\eqref{eq:key-ineq-1} and~\eqref{eq:key-ineq-3}.

It follows from inequality~\eqref{eq:key-ineq-3} that
\begin{equation}
\label{eq:key-ineq-2}
\binom{\alpha}{d + 1} = \frac{\alpha - d}{d + 1} \binom{\alpha}{d} \leq \frac{\alpha - d}{n - d - 1} \left(\binom{n}{d + 1} - \binom{n}{d} + t\right) + \frac{(\alpha-d)\cdot C_{d}\binom{n}{d-2}}{\Gamma_{d}}.
\end{equation}

By combining inequalities~\eqref{eq:key-ineq-1} and \eqref{eq:key-ineq-2}, we obtain the following inequality:
$$\binom{n}{d + 1} - \binom{n}{d} - c_{1}n^{d-2} {-2\Gamma_d n^3 t} \leq \frac{\alpha - d}{n - d - 1} \left(\binom{n}{d + 1} - \binom{n}{d} + t\right) +\frac{(\alpha-d)\cdot C_{d}\binom{n}{d-2}}{\Gamma_{d}}.$$
Upon simplification, we arrive at:
\begin{equation}\label{eq:key-ineq-4}
c_{1}n^{d - 2} +  {3 \Gamma_d n^3 t + \frac{(\alpha-d)\cdot C_{d}\binom{n}{d-2}}{\Gamma_{d}}} \ge \frac{n - 1 - \alpha}{n - d - 1} \left(\binom{n}{d + 1} - \binom{n}{d}\right) \geq c_{3} (n - 1 - \alpha) n^d,
\end{equation}
where the term \(\frac{\alpha - d}{n - d - 1} \cdot t\) is absorbed into {\(2\Gamma_d n^3\cdot t\)}, with \(c_{2}'\) being a constant that depends only on \(d\), and \(c_{3} > 0\) is a constant depending on \(d\).

On the other hand, since $\binom{\alpha}{d}$ is a degree $d$ polynomial in $\alpha$, when $\alpha \leq n - 1$, we have
$$\binom{\alpha}{d} \geq \binom{n - 1}{d} - c_{4}(n - 1 - \alpha) n^{d - 1},$$
for some $c_{4}>0$ which depends on $d$.
Therefore, it follows from inequality~\eqref{eq:key-ineq-3} that
$$\binom{n}{d + 1} - \binom{n}{d} + t \geq \frac{n - d - 1}{d + 1} \binom{n - 1}{d} - c_{5}(n-1-\alpha)n^{d} -\frac{C_{d}\binom{n}{d-2}(n-d-1)}{\Gamma_{d}},$$
for some $c_{5}>0$ which depends on $d$. Since $\frac{n - d - 1}{d + 1} \binom{n - 1}{d} = \binom{n - 1}{d + 1}$, we have
\begin{equation}\label{eq:key-ineq-5}
\begin{aligned}
t &\geq \binom{n - 1}{d + 1} + \binom{n}{d} - \binom{n}{d + 1} - c_{5}(n-1-\alpha)n^{d} - \frac{C_{d}\binom{n}{d-2}(n-d-1)}{\Gamma_{d}}\\ & = \binom{n-1}{d-1} - c_{5}(n-1-\alpha)n^{d} - \frac{C_{d}\binom{n}{d-2}(n-d-1)}{\Gamma_{d}}.
\end{aligned}
\end{equation}

Taking a linear combination of inequalities~\eqref{eq:key-ineq-4} and \eqref{eq:key-ineq-5}, we can see that 
\begin{equation}\label{equ:Final}
    \left({3 \Gamma_d n^3} + \frac{c_3}{c_5}\right) t + c_1 n^{d - 2} \geq \frac{c_3}{c_5} \binom{n - 1}{d - 1} - \frac{(\alpha-d)\cdot C_{d}\binom{n}{d-2}}{\Gamma_{d}} - \frac{c_{3}\binom{n}{d-2}(n-d-1)}{c_{5}\cdot\Gamma_{d}}.
\end{equation}
For sufficiently large $n$, we can take $\Gamma_d$ to be a constant significantly larger than $c_{3},c_{5}$ and $C_{d}$. Therefore, the inequality \eqref{equ:Final} implies that $t \geq c_d n^{d - 4}$ for some $c_{d}>0$, which depends only on $d$.

When \(\alpha > n-1\), the inequality~\eqref{eq:key-ineq-4} holds trivially. Therefore, by setting \(c_{4} = c_{5} = 0\) in the above argument, we deduce from~\eqref{eq:key-ineq-5} that \(t \geq c_{d}n^{d-1}\).  

In conclusion, by choosing \(\Gamma_{d}\) to be a sufficiently large constant depending only on \(d\), we obtain \(t \geq \Omega_{d}(n^{d-4})\). This completes the proof of~\cref{thm:MainPoly}.

\section{Generalizations of the Erd\H{o}s-Ko-Rado Theorem: Proof of~\cref{thm:Verification}}\label{sec:EKR}
As mentioned earlier, the case \(s=0\) corresponds to the famous Erd\H{o}s–Ko–Rado theorem. Next, we will prove the cases \(s=d\) and \(s=1\) separately.
\subsection{The case $s=d$}
We re-state the result here for convenience.
\begin{theorem}\label{thm:d=s}
    Let $n\ge d+1\ge 3$ be integers. If for every set $\mathcal{F}\subseteq\binom{[n]}{d+1}$, there exists some subset $B\subseteq F$ of size $d$ such that $F'\cap F\neq B$ for any $F'\in\mathcal{F}$, then $|\mathcal{F}|\le\binom{n-1}{d}$. Moreover, the equality holds if and only if $\mathcal{F}$ is a star of size $\binom{n-1}{d}$.
\end{theorem}

Indeed, \cref{thm:d=s} follows immediately from~\cref{claim:IndexD}. Here we present another proof via~\cref{thm:PartialShadow}, which might be of independent interest. 
\begin{proof}[Proof of~\cref{thm:d=s} via~\cref{thm:PartialShadow}]
Let $\mathcal{B}=\{B_F\mid F\in\mathcal{F}\}$. Note that we have $B_{F}\neq B_{F'}$ for any $F\neq F'\in \mathcal{F}$. Thus, $\abs{\mathcal{B}}=\abs{\mathcal{F}}$. 
    Then we apply~\cref{thm:PartialShadow} by setting $k=d$ and
    \begin{equation*}
        \mathcal{G}:=\{G\in\binom{[n]}{d}: G\subseteq F\ \textup{for\ some\ }F\in\mathcal{F}, G\neq B_F\}.
    \end{equation*}
    Observe that for any $F\in\mathcal{F}$, all of its $d$-shadows but $B_F$ are in $\mathcal{G}$. Thus, the set families $\mathcal{F},\mathcal{G}$ satisfies the condition in \cref{thm:PartialShadow}. Moreover, we know that $\mathcal{G}\cap \mathcal{B}=\emptyset$ since a $d$-shadow $A$ of $F\in\mathcal{F}$ which is not $B_F$ can not be the same as $B_{F'}$ for any $F'\in\mathcal{F'}$.

    Let $|\mathcal{F}|=\binom{x}{d}$ for some real number $x$. Ten we have
    \begin{equation*}
        \binom{n}{d}\ge |\mathcal{G}|+|\mathcal{B}|\ge\binom{x}{d-1}+\binom{x}{d}=\binom{x+1}{d},
    \end{equation*}
    which yields $x\le n-1$. This completes the proof.   
\end{proof}

\subsection{The case $s = 1$}
In this subsection, we prove \cref{thm:Verification} for the case \(s = 1\). Moreover, we establish a stronger stability result as follows:
\begin{theorem}
    \label{thm:s=1}
    Let \(d \geq 2\) be an integer, and let \(n\) be a sufficiently large integer relative to \(d\). Suppose \(\mathcal{F} \subseteq \binom{[n]}{d+1}\) satisfies the condition that for every set \(F \in \mathcal{F}\), there exists a subset \(B_F \subseteq F\) of size \(1\) such that \(F \cap F' \neq B_F\) for all \(F' \in \mathcal{F}\). Then \(\mathcal{F}\) is either a star with $|\mathcal{F}|\le\binom{n-1}{d}$, or \(|\mathcal{F}| \leq \binom{n-1}{d} - C_d n^{d-1}\) for some constant \(C_d > 0\).
\end{theorem}
\begin{proof}[Proof of~\cref{thm:s=1}]
    
We start with the following strengthening of~\cref{thm:HiltonMilner}.
\begin{lemma}\label{lem:genHM}
Let \(d \geq 2\) and \(n \geq 2d\). Suppose \(\mathcal{H} \subseteq \binom{[n]}{d}\) is a \(d\)-uniform intersecting family with \(|\mathcal{H}| \geq 6n^{d-2}\). Then there exists an element \(\ell \in [n]\) and sets \(A_1, \dots, A_{d+1} \in \mathcal{H}\) such that:  
\begin{enumerate}
    \item[\textup{(1)}] \(\ell \in A\) for all \(A \in \mathcal{H}\), and
    \item[\textup{(2)}] \(A_1 \setminus \{\ell\}, A_2 \setminus \{\ell\}, \dots, A_{d+1} \setminus \{\ell\}\) are pairwise disjoint.  
\end{enumerate}
    \end{lemma}
    \begin{proof}
As $d(d-1)\leq 6(d-2)!$ and $d^2 < 5(d-1)!$, we have
$$d(d-1)\binom{n-1}{d-2} < \frac{d(d-1)n^{d-2}}{(d-2)!} \leq 6n^{d-2}.$$
Applying the mean value theorem to the function $\binom{x}{d-1}$, we also have 
$$\binom{n-1}{d-1}-\binom{n-1-d}{d-1} < \frac{d^2(n-1)^{d-2}}{(d-1)!} <5n^{d-2}.$$ 
Thus, it follows that $$\abs{\mathcal{H}}>\max\left(\binom{n-1}{d-1}-\binom{n-1-d}{d-1}+1,d(d-1)\binom{n-1}{d-2}\right).$$
By~\cref{thm:HiltonMilner}, since \(\abs{\mathcal{H}} > \binom{n-1}{d-1} - \binom{n-1-d}{d-1} + 1\), we conclude that \(I = \bigcap_{A \in \mathcal{H}} A \neq \emptyset\).  
If \(\abs{I} > 1\), then we must have  
\[
\abs{\mathcal{H}} \leq \binom{n-\abs{I}}{d-\abs{I}} \leq \binom{n-2}{d-2},
\]  
which contradicts the assumption on \(\abs{\mathcal{H}}\). Thus, we may assume \(I = \{\ell\}\).  
Next, let \(r\) be the largest integer such that there exist sets \(A_1, \dots, A_r \in \mathcal{H}\) where \(A_1 \setminus \{\ell\}, \dots, A_r \setminus \{\ell\}\) are pairwise disjoint.  
If \(r \leq d\), then by the definition of \(r\), every \(A \in \mathcal{H}\) contains \(\ell\) and intersects \(\bigcup_{i=1}^r \big(A_i \setminus \{\ell\}\big)\). It follows that  
\[
\abs{\mathcal{H}} \leq r(d-1) \binom{n-1}{d-2} \leq d(d-1) \binom{n-1}{d-2},
\]  
which again contradicts the assumption on \(\abs{\mathcal{H}}\). Therefore, we must have \(r \geq d+1\), as required.
\end{proof}
For each \(a \in [n]\), define  
\[
\mathcal{F}_a := \{F \in \mathcal{F} : B_F = \{a\}\}.
\]  
We say that \(a\) is \emph{good} if \(\abs{\mathcal{F}_a} \geq 6n^{d-2}\); otherwise, \(a\) is called \emph{bad}. Define  
\[
G := \{a \in [n] \mid \text{\(a\) is good}\}.
\]  
Now consider any \(a \in G\). Observe that  
\[
\mathcal{X}_a = \{F \setminus \{a\} \mid F \in \mathcal{F}_a\}
\]  
is a \(d\)-uniform intersecting family of size at least \(6n^{d-2}\). By applying \cref{lem:genHM} to \(\mathcal{X}_a\), there exists an element \(\ell \in [n]\) and sets \(A_1, \dots, A_{d+1} \in \mathcal{X}_a\) such that:  
\begin{enumerate}
    \item[\textup{(1)}]\(\ell \in A\) for all \(A \in \mathcal{X}_a\), and
    \item[\textup{(2)}]\(A_1 \setminus \{\ell\}, \dots, A_{d+1} \setminus \{\ell\}\) are pairwise disjoint.
\end{enumerate}
 Finally, we set \(f(a) := \ell\).
 \begin{claim}\label{claim:a_to_fa}
For any $a \in G$ and $F\in \mathcal{F}$, if $F$ contains $a$, then $F$ must also contain $f(a)$. 
 \end{claim}
\begin{poc}
    Assume for the sake of contradiction that there exists $F\in\mathcal{F}$ with $a\in F$ and $f(a)\notin F$. Since $\abs{F\setminus\{a\}}=d$, we know that there must exist an index $i \in [d + 1]$ such that $F\cap A_i=\emptyset$, where $A_i$ is defined before the claim. However, this contradicts with the choice of $B_{A_i \cup \{a\}}=\{a\}$.
\end{poc}
\begin{claim}\label{claim:NoManyGoodElements}
If there exist some $a\in G$ and positive integer $k$ such that $a,f(a),f(f(a)),\dots, f^{(k)}(a)$ are all good and $f^{(k)}(a) = a$, then $\abs{\mathcal{F}}\leq \binom{n-1}{d} - c_{d}n^{d-1}$ for some constant $c_{d}>0$, 
\end{claim}
\begin{poc}
We may assume, without loss of generality, that \(k\) is the smallest positive integer such that \(a, f(a), \dots, f^{(k-1)}(a)\) are all distinct. Since \(f(a) \neq a\) by definition, it follows that \(k \geq 2\). To bound the number of sets \(F \in \mathcal{F}\), we consider whether \(F\) intersects \(\{a, f(a), \dots, f^{(k-1)}(a)\}\) or not. 
First, if \(F\) does not intersect \(\{a, f(a), \dots, f^{(k-1)}(a)\}\), then by \cref{thm:FranklPach}, the number of such sets is at most \(\binom{n-k}{d}\), as any subfamily of \(\mathcal{F}\) has VC-dimension at most \(d\). On the other hand, if \(F\) intersects \(\{a, f(a), \dots, f^{(k-1)}(a)\}\), then by \cref{claim:a_to_fa}, \(F\) must contain the entire set \(\{a, f(a), \dots, f^{(k-1)}(a)\}\). The number of such sets is at most \(\binom{n-k}{d+1-k}\).
Combining these bounds, if \(k \geq 3\), we have
\[
|\mathcal{F}| \leq \binom{n-3}{d} + \binom{n-3}{d-2} < \binom{n-2}{d} = \binom{n-1}{d} - \binom{n-2}{d-1}.
\]
When \(k = 2\), applying \cref{thm:GeneralD}, we obtain
\[
|\mathcal{F}| \leq \binom{n-3}{d} + C_d n^{d-1-\frac{1}{4d-2}} + \binom{n-2}{d-1} = \binom{n-1}{d} - c n^{d-1},
\]
for some constant \(c > 0\). This finishes the proof.
\end{poc}

By~\cref{claim:NoManyGoodElements}, we may assume that for all good \(a\), the sequence \(a, f(a), f(f(a)), \dots\) terminates at a bad element \(f^{(k)}(a)\) for some \(k \geq 1\) (otherwise, the proof is complete); we denote this bad element by \(b(a)\). From the definition and \cref{claim:a_to_fa}, if \(F \in \mathcal{F}\) contains \(a \in G\), then \(F\) also contains \(b(a)\). Let \[B = \{b \in [n] : b = b(a) \text{ for some good } a\},\] and define \[\mathcal{F}_{\mathrm{good}} = \{F \in \mathcal{F} : F \text{ contains at least one good element}\}.\]

For any  $F\in\mathcal{F}\setminus\mathcal{F}_{\mathrm{good}}$, the element in $B_F$ lies in $[n] \setminus G$. By the definition of $G$, we have
$$\abs{\mathcal{F}\setminus\mathcal{F}_{\mathrm{good}}}\leq \sum_{v \in [n] \setminus G} \abs{\mathcal{F}_v} \leq 6(n - \abs{G})n^{d-2}.$$ 
We now prove the following key claim, which consists of three upper bounds on the size of $\mathcal{F}_{\mathrm{good}}$.

\begin{claim}\label{claim:Fgood}
We have the following upper bounds on $|\mathcal{F}_{\textup{good}}|$.
\begin{enumerate}
    \item[\textup(1)] We always have 
    \[\abs{\mathcal{F}_{\mathrm{good}}}\leq \sum_{i=1}^d\binom{\abs{G}}{i}\binom{n-1-\abs{G}}{d-i}=\binom{n-1}{d}-\binom{n-1-\abs{G}}{d}.\] 
    
    \item[\textup(2)]  If $\abs{B}>1$, then we have
    \[\abs{\mathcal{F}_{\mathrm{good}}} \leq \binom{n-1}{d}-\binom{n-1-\abs{G}}{d}-\binom{\abs{G}-1}{d-1}.\] 

    \item[\textup(3)]  If $\abs{B}=1$ and there exists $F_0\in\mathcal{F}$ such that $F_0\subseteq [n]\setminus(G\cup B)$, then we have
    \[\abs{\mathcal{F}_{\mathrm{good}}} \leq \binom{n-1}{d}-\binom{n-1-\abs{G}}{d}-\binom{\abs{G}}{d-1}.\] 
\end{enumerate}
\end{claim}
\begin{poc}
    For (1): We can count the number of sets \(F \in \mathcal{F}_{\mathrm{good}}\) based on the size of \(F \cap G\). Suppose \(\abs{F \cap G} = i \geq 1\). There are at most \(\binom{\abs{G}}{i}\) ways to select the \(i\)-element subset \(F \cap G\). After fixing \(A = F \cap G\), we then choose any element \(a \in A\). We know that \(b(a) \in F \setminus G\), and there are at most \(\binom{n-1-\abs{G}}{d-i}\) ways to pick the remaining elements of \(F \setminus G\) from \( [n] \setminus (G \cup \{b(a)\})\).

Thus, the size of \(\mathcal{F}_{\mathrm{good}}\) is at most
\[
\sum_{i=1}^d \binom{\abs{G}}{i} \binom{n-1-\abs{G}}{d-i} = \binom{n-1}{d} - \binom{n-1-\abs{G}}{d},
\]
by Vandermonde's identity. This completes the proof of the first part.

For (2): Note that if \(\abs{B} > 1\), and \(A \subseteq G\) is a subset of size \(d\) containing any pair \(a, a' \in A\) such that \(b(a) \neq b(a')\), then it is impossible to have \(F \cap G = A\) for any \(F \in \mathcal{F}\). This is because, in that case, \(\{b(a), b(a')\} \subseteq F \setminus G\), which implies \(\abs{F} \geq d + 2\), a contradiction. Therefore, we conclude that  
\[
\abs{\{F : F \in \mathcal{F}_{\mathrm{good}}, \abs{F \cap G} = d\}} \leq \binom{\abs{G}}{d} - \bigg|\left\{ A \in \binom{G}{d} : \text{\(A\) contains \(a, a'\) with \(b(a) \neq b(a')\)} \right\}\bigg|.
\]  
Substituting this into the argument in part (1), we can improve the upper bound on \(\abs{\mathcal{F}_{\mathrm{good}}}\) to  
\[
\abs{\mathcal{F}_{\mathrm{good}}} \leq \sum_{i=1}^d \binom{\abs{G}}{i} \binom{n-1-\abs{G}}{d-i} - \bigg|\left\{ A \in \binom{G}{d} : \text{\(A\) contains \(a, a'\) with \(b(a) \neq b(a')\)} \right\}\bigg|.
\]
Fix any \(b \in B\), and let \(G_b = \{a \in G \mid b(a) = b\}\). Note that \(1 \leq \abs{G_b} \leq \abs{G} - 1\) since \(\abs{B} > 1\). We now estimate  
\[
\bigg|\left\{ A \in \binom{G}{d} \mid \text{\(A\) contains \(a, a'\) with \(b(a) \neq b(a')\)} \right\}\bigg| \geq \sum_{i=1}^{d-1} \binom{\abs{G_b}}{i} \binom{\abs{G} - \abs{G_b}}{d - i},
\]  
since we can choose \(i\) elements from \(G_b\) and \(d - i\) elements from \(G \setminus G_b\) to form \(A\). By applying the identity for binomial coefficients, we get  
\[
\sum_{i=1}^{d-1} \binom{\abs{G_b}}{i} \binom{\abs{G} - \abs{G_b}}{d - i} = \binom{\abs{G}}{d} - \binom{\abs{G} - \abs{G_b}}{d} - \binom{\abs{G_b}}{d} \geq \binom{\abs{G}}{d} - \binom{\abs{G} - 1}{d} = \binom{\abs{G}-1}{d-1},
\]  
This completes the proof of part (2).

For (3): Assume \(B = \{b\}\) and \(B_{F_0} = \{v\}\). For any \((d-1)\)-set \(A \subseteq G\), we know that \(A \cup \{b, v\}\) cannot lie in \(\mathcal{F}\) because \((A \cup \{b, v\}) \cap F_0 = \{v\} = B_{F_0}\). Therefore, we have  
\[
\big|\{ F : F \in \mathcal{F}_{\mathrm{good}}, \abs{F \cap G} = d - 1 \}\big| \leq \binom{\abs{G}}{d - 1} (n - 1 - \abs{G}) - \binom{\abs{G}}{d - 1}.
\]  
Substituting this into the argument from part (1), we obtain  
\[
\abs{\mathcal{F}_{\mathrm{good}}} \leq \sum_{i=1}^d \binom{\abs{G}}{i} \binom{n-1-\abs{G}}{d-i} - \binom{\abs{G}}{d-1},
\]  
as desired. This finishes the proof.
\end{poc}
We are ready to complete the proof of \cref{thm:s=1}. We proceed by casework on the sizes of $G$ and $B$. Let $C^*_d \in (0, 1)$ be any constant such that $\frac{(C^*_d)^{d-1}}{(d-1)!}> 100(1-C^*_d)$.
\begin{Case}
    \item If $\abs{G}\leq C^*_d\cdot n$, since $n$ is sufficiently large, by \cref{claim:Fgood}(1), we have
    \[\abs{\mathcal{F}}\leq \binom{n-1}{d}-\binom{n-1-\abs{G}}{d}+6(n - \abs{G})n^{d-2}< \binom{n-1}{d} - C_{1}n^{d}\]
   for some constant $C_{1}$ which only depends on $d$.
    \item If $\abs{G}\geq C^*_d\cdot n$ and $\abs{B}>1$, then when $n$ is sufficiently large, we have $\binom{\abs{G}-1}{d-1}\geq \frac{\abs{G}^{d-1}}{2(d-1)!}$. Thus~\cref{claim:Fgood}(2) implies that
    \begin{align*}
        \abs{\mathcal{F}}\leq &\binom{n-1}{d}-\binom{n-1-\abs{G}}{d}-\binom{\abs{G}-1}{d-1}+6(n-\abs{G})n^{d-2}\\
        \leq& \binom{n-1}{d}- \left(\frac{(C^*_d)^{d-1}}{2(d-1)!} + 6(C^*_d-1)\right)n^{d-1}\\
        \leq& \binom{n-1}{d} - 6(1 - C^*_d)n^{d - 1}.
    \end{align*}
    \item If $\abs{G}\geq C^*_d\cdot n$, $\abs{B}=1$, and there exists $F_0\in\mathcal{F}$ such that $F_0\subseteq [n]\setminus(G\cup B)$. When $n$ is sufficiently large, we have $\binom{\abs{G}}{d-1}\geq \frac{\abs{G}^{d-1}}{2(d-1)!}$. Thus, \cref{claim:Fgood}(3) implies that
    \begin{align*}
        \abs{\mathcal{F}}\leq &\binom{n-1}{d}-\binom{n-1-\abs{G}}{d}-\binom{\abs{G}}{d-1}+6(n-\abs{G})n^{d-2}\\
        \leq& \binom{n-1}{d}- \left(\frac{(C^*_d)^{d-1}}{2(d-1)!} + 6(C^*_d-1)\right)n^{d-1}\\
        \leq& \binom{n-1}{d} - 6(1 - C^*_d)n^{d - 1}.
    \end{align*}
    \item Finally, suppose that $\abs{B}=1$ and there does not exist $F\in\mathcal{F}$ such that $F\subseteq [n]\setminus(G\cup B)$. We claim that every $F\in \mathcal{F}$ contains $B$. Indeed, if $F \in \mathcal{F}_{\mathrm{good}}$, then there is some $a \in G$ such that $a \in F$, thus $B=\{b(a)\} \subseteq F$ by \cref{claim:a_to_fa} and the definition of $b(a)$. Otherwise, since $F \cap G=\emptyset$, we must have $B \subseteq F$. We conclude that $\mathcal{F}$ is trivially intersecting and $\abs{\mathcal{F}}\leq \binom{n-1}{d}$.
\end{Case}
Based on the detailed analysis of the above four cases, we  conclude that \(\mathcal{F}\) is either trivially intersecting, or \(|\mathcal{F}|\leq \binom{n-1}{d} - C_{d} n^{d-1}\) for some \(C_{d}>0\). This finishes the proof. 
\end{proof}

\section{Concluding remarks}\label{sec:conclusion}
Although it still seems very hard to resolve~\cref{question}, we believe the value $\binom{n-1}{d}+\binom{n-4}{d-2}$ should be the correct answer and our proof provides some new insights. Here we discuss more about the connection between the best-known construction of Mubayi-Zhao in~\cite{2007JAC} (which generalizes the original construction of Ahlswede and Khachatrian~\cite{1997CombFan}) and our proof of~\cref{thm:GeneralD}.

We first recall the construction of $\mathcal{F}$ with $|\mathcal{F}|=\binom{n-1}{d}+\binom{n-4}{d-2}$ in~\cite{2007JAC}. Let $\cG_1, \cG_2 \subseteq \binom{[n] \setminus \{1, 2\}}{d}$ satisfy the following three conditions:
\begin{enumerate}
    \item $\cG_1 \cup \cG_2 = \binom{[n] \setminus \{1, 2\}}{d}$.
    \item $\cG_1 \cap \cG_2 = \{F \in \binom{[n] \setminus \{1, 2\}}{d}: \{3, 4\} \subseteq F\}$.
    \item $\cG_1 \supseteq \{F \in \binom{[n] \setminus \{1, 2\}}{d}: 3 \in F\}$, $\cG_2 \supseteq \{F \in \binom{[n] \setminus \{1, 2\}}{d}: 4 \in F\}$.
\end{enumerate}
Note that
$$\binom{n - 2}{d - 1} + \abs{\cG_1} + \abs{\cG_2} = \binom{n - 2}{d - 1} + \binom{n - 2}{d} + \binom{n - 4}{d - 2} = \binom{n - 1}{d} + \binom{n - 4}{d - 2}.$$
It was claimed in~\cite{2007JAC} that the number of such choices of $(\cG_{1},\cG_{2})$ is $\frac{P(n-4,d)}{2}$, where $P(n-4,d)$ is the number of non-isomorphic $d$-uniform hypergraphs on $n-4$ vertices. After selecting the set systems \(\mathcal{G}_1\) and \(\mathcal{G}_2\), Mubayi and Zhao~\cite{2007JAC} constructed \(\mathcal{F} \subseteq \binom{[n]}{d+1}\) using four types of pairs \((F_k, B_k)\).
\begin{enumerate}
    \item Sets $F_k$ in $\binom{[n]}{d + 1}$ containing $\{1, 2\}$, with $B_k = F_k \setminus \{1, 2\}$ and $\abs{B_k} = d - 1$.
    \item Sets $F_k$ in $\{1\} \cup (\cG_1 \setminus \cG_2)$ with $B_k = F_k \setminus \{1\}$ and $\abs{B_k} = d$.
    \item Sets $F_k$ in $\{2\} \cup (\cG_2 \setminus \cG_1)$ with $B_k = F_k \setminus \{2\}$ and $\abs{B_k} = d$.
    \item Sets $F_k$ of the form $\{1\} \cup (\cG_1 \cap \cG_2)$ and $\{2\} \cup (\cG_1 \cap \cG_2)$, with $B_k = F_k \setminus \{1, 4\}$ if $1 \in F_k$ and $B_k = F_k \setminus \{2, 3\}$ if $2 \in F_k$. Note that for each set $F_{k}$ of this type, $\abs{B_k} = d - 1$.
\end{enumerate}
For this example, we have the small transversal set $J = \{1, 2\}$. Then some calculations give
\begin{enumerate}
    \item $\cT_J^1 = \emptyset$.
    \item $\cT_J^2$ contains the sets of Type 2 and 3.
    \item $\cT_J^3$ contains the sets of Type 4.
    \item $\cT_J^4 = \emptyset$.
    \item $\cT_J^5$ contains the sets of Type 1.
    \item $\cT_J^6 = \emptyset$.
\end{enumerate}
Note that
$$\cG_J = \cG_1 \cup \cG_2 = \binom{[n] \setminus J}{d}.$$
So we have $K = 0$. We can compute that
$$\abs{\cT_J^1} = 0,$$
$$\abs{\cT_J^2 \cup \cT_J^3} = \abs{\cG_1} + \abs{\cG_2} = \binom{n - 2}{d} + \binom{n - 4}{d - 2} = \binom{n - \abs{J}}{d} + \binom{n - 4}{d - 2},$$
$$\abs{\cT_J^4 \cup \cT_J^5} = \binom{n - 2}{d - 1} = (\abs{J} - 1) \binom{n - \abs{J}}{d - 1},$$
$$\abs{\cT_J^6} = 0.$$
Since $K=0$, this yields that~\cref{claim:T2T3} and~\cref{claim:T4T5} are sharp, which provides some evidence that the value $\binom{n-1}{d}+\binom{n-4}{d-2}$ is the correct answer.

As for~\cref{conj:TrueGeneralization}, fully resolving it remains challenging. In~\cref{thm:Verification}, we verify the special cases where $s \in \{0, 1, d\}$. In particular, for the case \(s = d\), we present two distinct types of proofs: one based on a double counting argument and the other utilizing an advanced tool introduced in recent work~\cite{2023partialShadow}. Furthermore, we demonstrate that when \(s = d\), the maximum size \(\binom{n-1}{d}\) can only be achieved if \(\mathcal{F}\) is a star of maximum size. A natural question arises as to whether a stronger stability result exists for~\cref{thm:d=s}, akin to the Hilton-Milner Theorem for the Erd\H{o}s-Ko-Rado theorem. Specifically, whether the size of \(\mathcal{F}\) would be significantly smaller than \(\binom{n-1}{d}\) if it is not a star. However, such a result does not hold. To illustrate this, consider \(A = \{2, 3, 4, \ldots, d, d+1, d+2\}\) and let \(\mathcal{C}\) consist of all \(d\)-subsets of \(A\). Then, the \((d+1)\)-uniform set system  
\[
\mathcal{F} := \{A\} \bigcup \bigg\{G \cup \{1\}: G \in \binom{[n] \setminus \{1\}}{d} \setminus \mathcal{C}\bigg\} 
\]  
satisfies the condition in~\cref{thm:d=s}, and \(\mathcal{F}\) is not a star. Nevertheless, \(|\mathcal{F}| = \binom{n-1}{d} - d \), which is very close to \(\binom{n-1}{d}\). 

Fortunately, as we proved in~\cref{thm:s=1}, such a strong stability result does exist when \(s=1\). From another perspective, attempting to extend the aforementioned construction indeed reduces \(|\mathcal{F}|\) by approximately \(n^{d-1}\). Therefore, we conjecture that if~\cref{conj:TrueGeneralization} holds, there might also exist a similar stability result. Specifically, for general \(s\), if \(\mathcal{F}\) is not a star, then its size would not exceed \(\binom{n-1}{d} - Cn^{d-s}\) for some positive constant $C$.

Nevertheless, we contend that a complete resolution of the entire~\cref{conj:TrueGeneralization} presents significant challenges and remains a highly nontrivial task. Even to completely settle the case \(d = s-1\) of this conjecture seems to require new ideas.

\bibliographystyle{abbrv}
\bibliography{FranklPach}
\end{document}